\documentclass[reqno,a4paper,11pt,oneside,final]{amsart}
 \UseRawInputEncoding
\usepackage{fixltx2e}
\usepackage{float}
\usepackage{amssymb,amsmath,amsthm}
\usepackage{mathtools}
\usepackage{fixmath}
\usepackage{pgfplots}
\usepackage{braket}
\usepackage{caption}
\usepackage{showkeys}
\usepackage{caption}
\usepackage{subcaption}
\usepackage{enumitem}
\usepackage{units}
\usepackage{xcolor}
\usepackage{booktabs}
\usepackage[english]{babel}
\usepackage[T1]{fontenc}
\usepackage{lmodern}
\usepackage{algorithm}
\usepackage{algorithmic}
\usepackage[final,stretch=10]{microtype}
\usepackage[colorlinks=true,linkcolor=blue!70!black,urlcolor=black,citecolor=green!60!black,pdftex]{hyperref}
\hypersetup{final} 

\newcommand{\N}{\mathbb{N}}
\newcommand{\R}{\mathbb{R}}
\newcommand{\de}{\mathrm{d}}

\newcommand{\M}{\mathcal{M}}
\newcommand{\Cc}{\mathcal{C}}
\newcommand{\Moc}{\M(I;\mathbb{R}^d)}
\newcommand{\Kb}{\mathcal{K}}

\newcommand{\BV}{\operatorname{BV}(I;\mathbb{R}^d)}
\newcommand{\bvnorm}[1]{\|#1\|_{\operatorname{BV}}}

\newcommand{\lonenorm}[1]{\|#1\|_{L^1}}

\newcommand{\optu}{\bar{u}}
\newcommand{\optom}{\bar{\omega}}

\newcommand{\optq}{\bar{q}}

\newcommand{\om}{\omega}

\newcommand{\Lzwo}{L^2(I;\mathbb{R}^d)}
\newcommand{\lzwo}[2]{\left(#1,#2\right)_{L^2}}

\newcommand{\mnorm}[1]{\|#1\|_{\mathcal{M}}}

\newcommand{\cnorm}[1]{\|#1\|_{\mathcal{C}}}
\newcommand{\ynorm}[1]{\|#1\|_{Y}}
\newcommand{\tr}[1]{tr(C(1)^{-1})}

\newcommand{\eps}{\varepsilon}

\newcommand{\vertiii}[1]{{\left\vert\kern-0.25ex\left\vert\kern-0.25ex\left\vert #1
    \right\vert\kern-0.25ex\right\vert\kern-0.25ex\right\vert}}

\makeatletter
\newcommand*\rel@kern[1]{\kern#1\dimexpr\macc@kerna}
\newcommand*\widebar[1]{%
  \begingroup
  \def\mathaccent##1##2{%
    \rel@kern{0.8}%
    \overline{\rel@kern{-0.8}\macc@nucleus\rel@kern{0.2}}%
    \rel@kern{-0.2}%
  }%
  \macc@depth\@ne
  \let\math@bgroup\@empty \let\math@egroup\macc@set@skewchar
  \mathsurround\z@ \frozen@everymath{\mathgroup\macc@group\relax}%
  \macc@set@skewchar\relax
  \let\mathaccentV\macc@nested@a
  \macc@nested@a\relax111{#1}%
  \endgroup
}
\makeatother

\numberwithin{equation}{section}

\definecolor{darkred}{rgb}{.7,0,0}

\definecolor{green}{rgb}{0,0.7,0}

\theoremstyle{plain}
\newtheorem{theorem}{Theorem}
\newtheorem{prop}[theorem]{Proposition}
\newtheorem{lemma}[theorem]{Lemma}
\newtheorem{coroll}[theorem]{Corollary}
\theoremstyle{definition}

\newtheorem{assumption}{Assumption}
\theoremstyle{remark}
\newtheorem{remark}{Remark}

\begin{document}
\title[]{A fast Primal-Dual-Active-Jump method for minimization in $\operatorname{BV}((0,T);\R^d)$}

\pagestyle{myheadings}
\author[]{Philip Trautmann$^{\tt1}$ \and Daniel Walter$^{\tt2}$}

\thanks{\vspace{-1em}\newline\noindent
$^{\tt1}$ Institute for Mathematics and Scientic Computing,
 Karl-Franzens-Universit\"at,
 Heinrichstr. 36, 8010 Graz, Austria. The author was supported by the ERC advanced grant 668998 (OCLOC) under the EU's H2020 research program.
 \newline
 $^{\tt2}$ Johann Radon Institute for Computational and Applied Mathematics,
  Altenbergerstr. 69, 4040 Linz, Austria,~\smaller\tt  daniel.walter@oeaw.ac.at
}

\maketitle


\begin{abstract}
We analyze a solution method for minimization problems over a space of $\R^d$-valued functions of bounded variation on an interval~$I$. The presented method relies on piecewise constant iterates. In each iteration the algorithm alternates between proposing a new point at which the iterate is allowed to be discontinuous and optimizing the magnitude of its jumps as well as the offset. A sublinear~$\mathcal{O}(1/k)$ convergence rate for the objective function values is obtained in general settings. Under additional structural assumptions on the dual variable this can be improved to a locally linear rate of convergence~$\mathcal{O}(\zeta^k)$ for some~$\zeta <1$. Moreover, in this case, the same rate can be expected for the iterates in~$L^1(I;\R^d)$.
\end{abstract}
\section{Introduction}
We consider minimization problems of the form
\begin{align}\label{def:bvprob}
\min_{u\in \BV} j(u)= \lbrack F(Ku) + \beta \mnorm{u'} \rbrack.\tag{$\mathcal{P}$}
\end{align}
where the minimizer is sought for in the space of~$\R^d$-valued functions of bounded variation on an interval~$I=(0,T)$. Here~$K$ denotes a linear and continuous operator mapping to a Hilbert space of observations~$Y$ and~$F$ is assumed to be a convex smooth loss function. Given~$\beta>0$, the second term in the objective functional penalizes the total variation norm of the distributional derivative~$u'$. It is well known that such a penalization favours minimizers which only change their values at a finite number of time points. This structural property of~\eqref{def:bvprob} makes it appealing for a variety of practical applications. For example, we point out PDE constraint optimal control problems,~\cite{CasasBV, EngelBV,NeitzelBV} and the denoising of scalar signals,~\cite{rudin,little}. For the precise functional analytic setting we refer to Section~\ref{sec:optproblem}.
\subsection{Contribution}
The aim of this paper is to analyze a simple yet efficient iterative solution algorithm for problem~\eqref{def:bvprob}. It relies on the identification of~$u \in \BV$ with its distributional derivative~$u'$ and the mean values of its components~$a_u  \in \R^d$. The proposed method generates sequences of piecewise constant iterates~$u_k$ and active sets~$\mathcal{A}_{k}=\{(\mu^k_i, v^k_i)\}^{N_k}_{i=1}$ which store the jumps~$v^k_i$ of~$u_k$ as well as the associated magnitudes~$\mu_i$. By a "jump"~$v^k_i$ we refer to an atomic measure supported on a position~$t^k_i \in I$ together with a normalized direction~$\mathbf{v}^k_i \in \R^d$. Each iteration then proceeds in three phases: First we allow for an additional jump~$\widehat{v}^k$ in the iterate~$u_k$. The position and the direction of this new candidate jump are determined based on a dual variable~$p_k \in \mathcal{C}_0(I;\R^d)$. Subsequently we determine improved magnitudes for~\textit{all} jumps in the active set as well as a new vector of mean values by solving the finite-dimensional convex minimization problem
\begin{align*}
\min_{\mu_i \geq 0,a_u \in \R^d} \left \lbrack F(Ku)+ \beta \sum^{N_k +1}_{i=1} \mu_i \right \rbrack~ s.t.~ u'=\mu_{N_k +1} \widehat{v}^k+\sum^{N_k}_{i=1} \mu^k_i v^k_i.
\end{align*}
Finally the active set is updated by removing all jumps whose associated magnitude was set to zero. The theoretical contribution of the present manuscript is twofold. First we prove that the generated sequence~$u_k$ indeed converges, on subsequences, to minimizers of~\eqref{def:bvprob} and the functional values~$j(u_k)$ converge sublinearly to the minimum value. Second, under appropriate structural assumptions on the optimal dual variable, similar to~\cite{NeitzelBV,EngelBV}, we deduce the local linear convergence of~$j(u_k)$ and of the iterates $u_k$ with respect to the strict topology on~$\BV$.
\subsection{Related work}
The efficient algorithmic solution of\eqref{def:bvprob} is a delicate issue for a variety of reasons. On the one hand this is attributed to the appearance of the BV seminorm which makes the objective functional nonsmooth. Moreover~$j$ lacks coercivity with respect to~$u$ which is often a vital tool in the derivation of fast convergence result for minimization schemes. On the other hand we point out that~$\BV$ is non-reflexive. Many well-studied algorithms for non-smooth optimization rely on the reflexive structure of the underlying space and thus donot yield direct extensions to the problem at hand.

A first straightforward approach to circumventing the aforementioned difficulties consists of discretizing the space~$\BV$ in~\eqref{def:bvprob}. More in detail, instead of minimizing over all~$u\in \BV$, one could only consider piecewise constant~$u_h$ that solely jump in the nodes~$0=t_0 < t_1< \cdots < t_{N_h}=T$ of a partition of~$I$. This reduces~\eqref{def:bvprob} to a finite dimensional convex minimization problem with a nonsmooth group sparsity regularization term,~\cite{huang}. The solution of discretized 1D BV problems has been addresses e.g. in~\cite{condattaut,jimnezBV,wahlbergBV,villeBV}. Nonetheless, such reasoning often leads to algorithms that exhibit~\textit{mesh-dependency} meaning that their convergence behaviour critically depends on the partition of~$I$ and might degenerate as~$N_h \rightarrow \infty$. To mitigate these effects, a second line of works,~see e.g.~\cite{ClasonBV, HafeBV}, proposes the regularization of~\eqref{def:bvprob} by adding~$(\eps/2) \|u'\|^2_{L^2}$ for~$0 <\eps << 1$ and minimizing for~$u$ in the Sobolev space~$H^1(I;\R^d)$. Since the total variation norm of~$u'$ remains present in objective functional, the derivative of minimizers can still be expected to exhibit sparsity i.e. its support is small. However, due to the Sobolev seminorm penalty, minimizers cannot be piecewise constant if~$\eps> 0$. For this reason algorithmic approaches based on regularization are usually accompanied by a path-following strategy for~$\eps \rightarrow 0$ which requires additional analysis.

If~\eqref{def:bvprob} is restricted to mean-value free BV functions, it can be equivalently reformulated as minimization problem over the space of~$\R^d$-valued vector measures. Over the past years there has been an increasing body of work on the efficient solution of such problems using exchange type algorithms,~\cite{WalterPDAP,Denoyelle,Flinthexchange,Bredies,Boydsparse}, which rely on iterates comprised of finitely many Dirac Delta functionals. These alternate between proposing a new Dirac Delta (i.e. a "jump" in our terminology) and approximately solving finite-dimensional convex and/or nonconvex subproblems to achieve sufficient descent. Most recently, linear convergence of such methods relying on convex subproblems has been addressed in~\cite{Flinthexchange}, for~$d=1$, and~\cite{WalterPDAP}, for the general vector-valued case. Our approach is closest related to the earlier work~\cite{WalterPDAP} but differs in the treatment of the convex subproblems. More in detail while the present method relies solely on optimizing the magnitudes~$\mu_i$ in each iteration ($\# \mathcal{A}_k $ DOF), the linear convergence result of~\cite{WalterPDAP} also requires the optimization of the jump directions ($d\#\mathcal{A}_k $ DOF). Hence we obtain the same theoretical convergence guarantees while solving smaller subproblems. Let us also mention the finite step convergence results of~ \cite{Flinthexchange,Denoyelle}. However these require "point-moving" i.e. an additional step in which the jump positions are optimized. This constitutes a nonconvex problem and is therefore not considered in the present work. The idea of using exchange type methods for 1D BV penalties has previously been proposed in~\cite{BoydBV} together with a sublinear convergence result.

Finally we point out the denoising problem for a scalar signal~$y_d \in L^2(I)$. In our setting this corresponds to the case of~$d=1$,~$F= (1/2) \|\cdot-y_d\|^2_{L^2}$ and~$K= \operatorname{Id}$. For this particular instance of~\eqref{def:bvprob} the unique minimizer can be determined directly using a~\textit{taut-string-method}, see~e.g.~\cite{grasmairBV,KunischBV}. To the best of our knowledge this method does, however, not yield extensions to the case of a general observation operator~$K$ and the vector-valued case~$d>1$.

\subsection{Outline of the paper}
The relevant notation used throughout the paper is introduced in Section~\ref{sec:notation}. In Section~\ref{sec:optproblem} we equivalently reformulate~\eqref{def:bvprob} as minimization problem over the distributional derivative and the mean value of~$u$. Subsequently, this equivalence is used to derive first-order necessary and sufficient optimality conditions. A detailed description of the proposed solution algorithm for~\eqref{def:bvprob} can be found in Section~\ref{sec:algsol}. The convergence of the method is adressed in Section~\ref{sec:convergence}. Finally, Section~\ref{sec:num} finishes the paper with numerical experiments illustrating our theoretical findings.
\section{Notation \& definitions} \label{sec:notation}
In the following set~$I=(0,T)$ for some~$T>0$ and fix~$d\in \N$. Denote by~$(\cdot,\cdot)_{\R^d}$ the euclidean inner product on~$\R^d$ and let~$|\cdot|_{\R^d}$ denote the corresponding norm. By~$\mathcal{C}_0(I;\R^d)$ together with the usual supremum norm
\begin{align*}
\cnorm{\varphi}=\sup_{t \in I } |\varphi(t)|_{\R^d} \quad \forall \varphi \in \mathcal{C}_0(I;\R^d)
\end{align*}
we refer to the Banach space of~$\R^d$-valued continuous functions on~$I$ that vanish at its boundary. Its topological dual space is readily identified with the space of regular vector measures~$\mathcal{M}(I;\R^d)$. The corresponding duality pairing is denoted by~$\langle \cdot, \cdot\rangle$.
For example, if~$q$ is a discrete measure, i.e.~$q=\sum^N_{i=1} \mathbf{q}_i \, \delta_{t_i}$ where~$\mathbf{q}_i \in \R^d$ and~$\delta_{t_i}$ denotes the Dirac Delta functional supported on~$t_i \in I$, then
\begin{align*}
\langle \varphi , q \rangle= \sum^N_{i=1} (\varphi(t_i),\mathbf{q}_i)_{\R^d}.
\end{align*}
 The space~$\mathcal{M}(I;\R^d)$ is equipped with the canonical dual norm
\begin{align*}
\mnorm{q}= \sup_{\cnorm{\varphi}=1}\langle \varphi,q \rangle.
\end{align*}

We call~$u\in L^1(I;\R^d)$ a \emph{function of bounded variation} if its distributional derivative~$u'$ is representable by a an element of~$\M(I;\R^d)$ i.e.
\begin{align*}
\lzwo{ u}{\varphi'}= \langle \varphi, u'  \rangle \quad \forall \varphi \in \mathcal{C}^\infty_c(I;\R^d).
\end{align*}
The set of~$\R^d$-valued functions of bounded variation on~$I$ is now defined as
\begin{align*}
\BV = \left \{\,u \in L^1(I;\R^d)\;|\;\mnorm{u'}<\infty\,\right\}.
\end{align*}
Equipping~$\BV$ with the norm
\begin{align*}
\bvnorm{u}=\mnorm{u'}+\lonenorm{u} \quad \forall u \in \BV,
\end{align*}
where
\begin{align*}
\|u\|_{L^1}=\int_I |u(s)|_{\R^d}~\mathrm{d}s,
\end{align*}
makes it a Banach space which continuously embeds into~$L^p(I)$,~$p\in [1,\infty]$, the embedding being compact for~$p<\infty$. Given a function~$u \in L^1(I;\R^d)$, the vector of the mean values of its components is defined as
\begin{align*}
a_u= \frac{1}{T} \int_I u(t)~\de t,
\end{align*}
where integration has to be understood in the sense of Bochner.
From e.g.~\cite[Theorem~3.44]{ambrosio} we conclude the existence of constants~$C_1,~C_2>0$ with
\begin{align*}
C_1(|a_u|_{\R^d}+ \mnorm{u'}) \leq  \bvnorm{u} \leq C_2 (|a_u|_{\R^d}+ \mnorm{u'} ) \quad \forall u \in \BV.
\end{align*}
Following~e.g.~\cite[Remark~3.12]{ambrosio}~$\BV$ can be identified as the topological dual space of a separable Banach space. A sequence~$\{u_k\}_{k\in\N} \subset \BV$ is called weak* convergent in~$\BV$ with limit~$\optu$ if
\begin{align*}
\|u_k-\optu \|_{L^1} \rightarrow 0, \quad u'_k \rightharpoonup^* \optu' .
\end{align*}
Due to the sequential Banach-Alaoglu theorem every bounded sequence in~$\BV$ admits a weak* convergent subsequence.
Furthermore a weak* convergent sequence~$\{u_k\}_{k\in\N} \subset \BV$ is called convergent with respect to the strict topology on~$\BV$, or shortly strictly convergent, if additionally~$\mnorm{u'_k} \rightarrow \mnorm{\optu}$ holds. This is indicated by "$\rightharpoonup^s$". The strict topology on~$\BV$ is induced by the metric
\begin{align*}
d(u_1,u_2)= \|u_1-u_2\|_{L^1}+|\mnorm{u_1}-\mnorm{u_2}|.
\end{align*}
Last, given an open interval~$(t,1)$ for some~$t \in [0,1)$ its characteristic function is defined by
\begin{align*}
\chi_{t}=  \begin{cases} 0 & \text{on}~I \setminus (t,1) \\ 1& \text{else}  \end{cases}.
\end{align*}
There holds~$\chi_{t} \in \operatorname{BV}(I)$ with~$\chi'_t=\delta_t$,~$t>0$, and~$\chi'_0=0$, respectively.
\section{Optimization problem} \label{sec:optproblem}
The following assumptions concerning~\eqref{def:bvprob} are made throughout this paper.
\begin{assumption} \label{ass:problem}
In the following let~$Y$ be a Hilbert space with inner product~$(\cdot,\cdot)_Y$ and induced norm~$\|\cdot\|_Y$.
There holds:
\begin{itemize}
\item The operator~$K \colon L^2(I;\R^d)\to Y$ is linear and continuous.
\item The mapping~$F \colon Y \to\R$ is strictly convex and continuously Fr\'echet differentiable. The Fr\'echet derivative~$F'(y)\in\mathcal{L}(Y,\R)$ of~$F$ at~$y\in Y$ is identified with its Riesz representative~$\nabla F(y)  \in Y$ i.e.
\begin{align*}
F'(y) \delta y=(\nabla F(y), \delta y)_Y \quad \forall \delta y\in Y.
\end{align*}
\item The functional~$j \colon \BV \to \R $ in~\eqref{def:bvprob} is radially unbounded i.e.
\begin{align*}
\bvnorm{u_k} \rightarrow \infty \Rightarrow j(u_k) \rightarrow +\infty.
\end{align*}
\end{itemize}
\end{assumption}
Existence of solutions to~\eqref{def:bvprob} can be obtained using the direct method. Since the proof is fairly standard we omit it at this point.
\begin{prop} \label{prop:existence}
Let Assumption~\ref{ass:problem} hold. Then there exists at least one solution~$\bar{u}\in \BV$ to~\eqref{def:bvprob}.
\end{prop}
\subsection{Optimality conditions} \label{subsec:eqref}
The derivation of most subsequent results relies on an equivalent reformulation of~\eqref{def:bvprob} which will be introduced next. Define the linear and continuous operator
\begin{align}\label{def:Boperator}
B \colon \Moc \times \R^d \to \Lzwo, \quad (q, c) \mapsto \int_0^{\cdot} \de q- \frac{1}{T}\int_0^T~ \int_0^s \de q~\de s +c,
\end{align}
where integration hast to be understood in the sense of Bochner. We arrive at the following identification.
\begin{prop} \label{prop:equivalence}
For~all~$(q, c)\in \Moc \times \R^d$ we have~$B(q, c)\in \BV$.
The linear and continuous operator~$B \colon \Moc \times \R^d \to \BV$ from~\eqref{def:Boperator} is an isomorphism.
\end{prop}
\begin{proof}
The bounded invertibility of~$B$ is imminent noting that its inverse is given by the operator
\begin{align*}
B^{-1} \colon \BV \to \Moc \times \R^d  , \quad u \mapsto \left(u', a_u \right).
\end{align*}
\end{proof}
Loosely speaking, the previous result states that any function of bounded variation on~$I$ is uniquely characterized by its distributional derivative and mean values of its components. Thus~\eqref{def:bvprob} is equivalent to the sparse minimization problem
\begin{align} \label{def:sparsebvprop}
\min_{q \in \Moc,~c \in \R^d} \lbrack F(\mathcal{K}(q,c))+\beta \mnorm{q} \rbrack.
\end{align}
where we abbreviate~$\mathcal{K}=K \circ B$.

Next we characterize the adjoint operator~$B^*$. Consider the system of auxiliary ordinary differential equations
\begin{align} \label{def:auxode}
-\om''= \phi \quad \text{in}~(0, T), \quad \om'(0)=\om'(T)=0, \quad \int^T_{0} \om (t)~\de t=0,
\end{align}
where~$\phi\in L^2(I;\R^d) $ with~$a_\phi =0$. Clearly, this problem admits a unique solution~$\om \in H^2(I;\R^d) \hookrightarrow \Cc^1(I;\R^d)$ and~$\om' \in \Cc_0(I;\R^d)$.
\begin{lemma} \label{lem:charaofadjB}
The linear and continuous operator~$B$ from~\eqref{def:Boperator} is the Banach space adjoint of
\begin{align}
B^* \colon \Lzwo \to \mathcal{C}_0(I;\R^d)\times  \R^d, \quad \varphi \mapsto \left( \om', \int_0^T \varphi(s) ~\de s \right).
\end{align}
where~$\om \in \Cc^1(I;\R^d) $ fulfills~\eqref{def:auxode} for~$\phi= \varphi-a_\varphi$.
\end{lemma}
\begin{proof}
Obviously, the operator~$B^*$ is linearly and continuous.
Let~$\varphi \in \Lzwo$ and a pair~$(q,c)\in \Moc \times \R^d$ be given. We readily obtain
\begin{align*}
 \langle \om', q \rangle + \left(c, \int^T_0 \varphi(t) ~\de t \right) &= \int_0^T \left(\varphi(s), \int_0^s~\de q \right)~\de s-\left( a_\varphi, \int_0^T \int_0^s~\de q~\de s\right)+\left( c ,\int^T_0 \varphi(t) ~\de t \right)
\\&= \lzwo{\varphi}{B(q,0)}+\left(c, \int^T_0 \varphi(t) ~\de t \right)
\\&= \lzwo{\varphi}{B(q,c)}.
\end{align*}
Here we used~$B(q,0)\in \BV$ with~$B(q,0)'=q$ as well as the integration by parts in the second equality. This establishes the result.
\end{proof}
Combining the equivalence of~\eqref{def:bvprob} and~\eqref{def:sparsebvprop} as well as the characterization of $B^*$ we arrive at the following necessary and sufficient first order optimality conditions.
\begin{theorem} \label{thm:firstorderoptimality}
Let~$\optu \in \BV$ be given. Further define
\begin{align*}
\bar{p}(t)= \int_0^t K^* \nabla F(K \bar{u})(s)~\de s \in \mathcal{C}(\bar{I}; \R^d)
\end{align*}
Then~$\optu$ is an optimal solution to~\eqref{def:bvprob} if and only if
\begin{align}
\cnorm{\bar{p}}\in \begin{cases} \{\beta\} & \bar{u}' \neq 0  \\ [0,\beta] & \text{else} \end{cases}, \quad \bar{p}(T) =0
\end{align}
as well as
\begin{align} \label{eq:poshom}
\langle \bar{p}, \optu' \rangle= \beta \mnorm{\optu'}.
\end{align}
\end{theorem}
\begin{proof}
A~function~$\optu \in \BV$ is an optimal solution to~\eqref{def:bvprob} if and only if the pair
\begin{align*}
(\bar{q},\bar{c})=\left(\optu', \frac{1}{T} \int_0^T \optu(s) ~\de s\right)
\end{align*}
is a minimizer of~\eqref{def:sparsebvprop}. Since~$J$ is convex and~$F$ is Fr\'echet-differentiable, optimality of~$(\bar{q}, \bar{c})$ is equivalent to
\begin{align} \label{eq:optcond1}
(-\nabla F(\Kb(\bar{q},\bar{c})),\Kb(q-\bar{q},0))_{\Lzwo}+G(\mnorm{\bar{q}})\leq G(\mnorm{q}) , \quad \forall  q \in \Moc
\end{align}
as well as
\begin{align} \label{eq:optcond2}
(-\nabla F(\Kb(\optq,\bar{c})),\Kb(0, \delta c ))_{\Lzwo}=0 \quad \forall \delta c \in \R.
\end{align}
Let~$\optom \in \mathcal{C}^1(\bar{I};\R^d)$ denote the solution of~\eqref{def:auxode} for~$\varphi=- K^* \nabla F(K \bar{u})\in L^2(I)$.
Note that~$\Kb^*= B^*K^*$. Utilizing the characterization of~$B^*$, see Lemma~\ref{lem:charaofadjB}, as well as the definition of the convex subdifferential the conditions~\eqref{eq:optcond1} and~\eqref{eq:optcond2}, respectively, can be rewritten as
\begin{align} \label{eq:constiszero}
\optom' \in \beta \partial \mnorm{\bar{q}}, \quad \bar{p}(T)=\int_0^T - K^* \nabla F(K \bar{u}) (t)~\de t=0.
\end{align}
It is well known, that the subdifferential inclusion is equivalent to
\begin{align*}
\cnorm{\optom'}\in \begin{cases} \{\beta\} & \bar{q} \neq 0  \\ [0,\beta] & \text{else} \end{cases}, \quad \langle \optom',\bar{q}\rangle  = \beta \mnorm{\bar{q}}.
\end{align*}
Due to the fundamental theorem of analysis, there exists a vector~$c\in \R^d$ with
\begin{align*}
\optom'(t) &= \int_0^t K^* \nabla F(K \bar{u})(s)~\de s -t \int_0^T K^* \nabla F(K \bar{u})(s)~\de s +c \\
&=\int_0^t K^* \nabla F(K \bar{u})(s)~\de s +c
\end{align*}
for all~$t \in \bar{I}$. From~$\optom'(0)=0$ we deduce~$c=0$. Thus we conclude~$\bar{p}=\optom'$ on~$\bar{I}$.
Combining all the previous observations now finishes the proof.
\end{proof}
It is by now well-known that the extremality condition in~\eqref{eq:poshom} ensures the sparsity of~$\optu'$ if the dual variable~$\bar{p}$ only admits finitely many global extrema.
\begin{coroll} \label{coroll:suppcond}
Let~$\optu \in \BV$ be a minimizer of~\eqref{def:bvprob} and let~$\bar{p}$ be defined as in Theorem~\ref{thm:firstorderoptimality}. Assume that
\begin{align} \label{eq:suppcond}
 \{\bar{t}_i\}^N_{i=1}= \left \{\,t \in I\;|\;|\bar{p}(t)|_{\R^d}= \beta\,\right\}
\end{align}
for some~$N\in\N$ and~$\{\bar{t}_i\}^N_{i=1}\subset I$. Then~$\optu' \in\Moc$ is of the form
\begin{align*}
\optu'=\sum^N_{i=1} \bar{\mu}_i \bar{v}_i~\text{where}~\bar{v}_i=(\bar{p}(\bar{t}_i)/\beta) \delta_{\bar{t}_i}
\end{align*}
i.e.~$\bar{u}$ is piecewise constant on~$I$.
\end{coroll}
\begin{proof}
This can be proven analogously to~\cite[Corollary 6.25]{walter}.
\end{proof}
Finally we point out that the optimal observation~$\bar{y}\in Y$ in~\eqref{def:bvprob} and thus also the dual variable~$\bar{p} \in \mathcal{C}_0(I;\R^d)$, see Theorem~\ref{thm:firstorderoptimality}, are unique.
\begin{coroll}
Let~$\optu_1, \optu_2 \in \BV $ denote two minimizers to~\eqref{def:bvprob}. Moreover denote by~$\bar{y}_1 =K \optu_1,~\bar{y}_2 =K \optu_2$ and~$\bar{p}_1,\bar{p}_2 \in \mathcal{C}_0(I;\R^d)$ the associated observations and dual variables, see Theorem~\ref{thm:firstorderoptimality}, respectively. Then~$\bar{y}_1=\bar{y}_2$ and~$\bar{p}_1=\bar{p}_2$.
\end{coroll}
\begin{proof}
The uniqueness of the optimal observation, and thus also that of the dual variable, directly follows from the strict convexity of~$F$.
\end{proof}
\section{Algorithmic solution} \label{sec:algsol}
This section is devoted to the description of an efficient solution algorithm for~\eqref{def:bvprob}. The method we propose relies on the iterative update of a finite active set~$\mathcal{A}_k=\{\mu^k_i, v^k_i\}^{N_k}_{i=1}$ comprised of "jumps" $v^k_i \in \Moc$ and the associated "magnitudes"~$\mu^k_i > 0$. Each jump is of the form~$v^k_i= \mathbf{v}^k_i \delta_{t^k_i}$ for a position~$t^k_i \in I$  and a direction~$\mathbf{v}^k_i \in \R^d$,~$|\mathbf{v}^k_i|_{\R^d}=1$. Given an offset~$c^k\in \R^d$ the~$k-th$ iterate is defined as
\begin{align} \label{def:formofuk}
u_k=B \left( \sum^{N_k}_{i=1} \mu^k_i v^k_i, c^k \right).
\end{align}
If~$\mathcal{A}_k= \emptyset$, i.e.~$u^k=c^k \chi_0$, we adopt the convention~$N_k=0$.
We shortly describe the individual steps of the algorithm in the following. A summary can be found in Algorithm~\ref{alg:accgcg}. Given the current active set~$\mathcal{A}_k$ and iterate~$u_k$ we first compute the current dual variable~$p_k(\cdot)=\int^\cdot_0 K^*\nabla F(Ku_k)~\de s$ as well one of its global extrema~$\hat{t}_k \in I$. Next, assuming that~$\cnorm{p_k} >0$, see Proposition~\ref{prop:optimalitysub}, we define the new candidate jump~$\widehat{v}^k \coloneqq (p_k(\hat{t}_k)/\cnorm{p_k})\delta_{\hat{t}_k}$ and find improved jump heights~$\mu^{k+1/2} \in \R^{N_k +1}_+$ and a new offset~$c^{k+1} \in \R^d$ from solving
\begin{align} \label{def:subprobjumps}
\min_{(\mu,c) \in \R^{N_k +1}_+ \times \R^d} \left \lbrack F\left( \mathcal{K} \left( \mu_{N_k +1} \widehat{v}^k+ \sum^{N_k}_{i=1} \mu_i v^k_i, c \right) \right)+ \beta \sum^{N_k +1}_{i=1} \mu_i \right \rbrack. \tag{$\mathcal{P}_{\mathcal{A}_k}$}
\end{align}
This represents a finite-dimensional convex minimization problem with box constraints which can be tackled by a variety of efficient solution algorithms. Now the new jump is added to the active set and the jump heights are updated setting
\begin{align*}
\mathcal{A}_{k+1/2} \coloneqq (\mu^{k+1/2}_{N_k +1}, \widehat{v}^k) \cup \left\{(\mu^{k+1/2}_i,v^k_i)\right\}^{N_k}_{i=1}.
\end{align*}
Finally we prune the active set by removing all jumps whose associated jump magnitude was set to zero i.e.
\begin{align*}
\mathcal{A}_{k+1} \coloneqq \left\{\mu^{k+1}_i, v^{k+1}_i\right\}^{N_{k+1}}_{i=1}=\left\{\,(\mu,v)\in \mathcal{A}_{k+1/2}\;|\;\mu>0\,\right\}
\end{align*}
and increment the iteration counter~$k$ by one.

\begin{algorithm}
\begin{flushleft}
\hspace*{\algorithmicindent} \textbf{Input:} Active set~$\mathcal{A}_0=\{(\mu^0_i,v^0_i)\}^{N_0}_{i=1}$, iterate~$u_0=B \left( \sum^{N_0}_{i=1} \mu^0_i v^0_i, c^0 \right)$ \\
\hspace*{\algorithmicindent} \textbf{Output:} Minimizer~$\optu$ to~\eqref{def:bvprob}.
\end{flushleft}
\begin{algorithmic}
\STATE
\STATE 1. Find~$(\mu^{1/2},c^1)\in\R^{N_0}_+ \times \R^d$ by solving
\begin{align*}
\min_{(\mu,c) \in \R^{N_0}_+ \times \R^d} \left \lbrack F\left( \mathcal{K} \left( \sum^{N_0}_{i=1} \mu_i v^0_i, c \right) \right)+ \beta \sum^{N_0 }_{i=1} \mu_i \right \rbrack.
\end{align*}
\STATE 2. Prune active set and update iterate:
\begin{align*}
\mathcal{A}_1 =\left\{\,(\mu^{1}_i,v^1_i)\right\}^{N_1}_{i=1}=\left\{\,(\mu^{1/2}_i,v^0_i )\;|\;\mu^{1/2}_i>0\,\right\},~u_1=B \left( \sum^{N_0}_{i=1} \mu^1_i v^1_i, c^1 \right).
\end{align*}
\FOR {$k=1,2,\dots$}
\STATE
\STATE 3. Compute~$p_k \in \mathcal{C}_0(I;\R^d)$ and~$\hat{t}_k\in I$ with
\begin{align*}
p_k=\int^\cdot_0 K^* \nabla F(Ku_k)(s)~\de s, \quad |p_k(\hat{t}_k)|_{\R^d}=\cnorm{p_k}= \max_{t \in I} |p_k(t)|_{\R^d}.
\end{align*}
\IF{$\cnorm{p_k}\leq \beta$}
\STATE
\STATE 4. Terminate with~$ \bar{u}=u_k$ a minimizer to~\eqref{def:bvprob}.
 \STATE
\ENDIF
\STATE
\STATE 5. Find~$ (\mu^{k+1/2},c^{k+1})\in\R^{N_k +1}_+ \times \R^d$ from~\eqref{def:subprobjumps} for $\widehat{v}^k= (p_k(\hat{t}_k)/\cnorm{p_k}) \delta_{\hat{t}_k}$.
\STATE
\STATE 6. Update active set:
\begin{align*}
\mathcal{A}_{k+1/2} \coloneqq (\mu^{k+1/2}_{N_k +1}, \widehat{v}^k) \cup \left\{(\mu^{k+1/2}_i,v^k_i)\right\}^{N_k}_{i=1}.
\end{align*}
\STATE 7. Prune active set and update iterate
\begin{align*}
\mathcal{A}_{k+1} &\coloneqq \left\{\mu^{k+1}_i, v^{k+1}_i\right\}^{N_{k+1}}_{i=1}= \left\{\,(\mu,v)\in \mathcal{A}_{k+1/2}\;|\;\mu>0\,\right\},\\ u_{k+1}&=B \left( \sum^{N_{k+1}}_{i=1} \mu^{k+1}_i v^{k+1}_i, c^{k+1} \right) .
\end{align*}
and set~$k=k+1$.
\ENDFOR
\end{algorithmic}
\caption{Primal-dual-active-jump method (PDAJ) for~\eqref{def:bvprob}}
\label{alg:accgcg}
\end{algorithm}
We point out that the termination criterion of Algorithm~\ref{alg:accgcg} relies on the norm of~$p_k$. This is justified by the following proposition.
\begin{prop} \label{prop:optimalitysub}
Denote by
\begin{align*}
\mathcal{A}_k = \left\{(\mu^k_i,v^k_i)\right\}^{N_k}_{i=1}= \left\{(\mu^k_i,\mathbf{v}^k_i \delta_{t^k_i})\right\}^{N_k}_{i=1},~ u_k= B \left( \sum^{N_k}_{i=1} \mu^k_i v^k_i, c^k \right)
\end{align*}
the sequences of active sets and iterates generated by Algorithm~\ref{alg:accgcg}. Moreover set~$p_k(\cdot)= \int^\cdot_0 K^* \nabla F(Ku_k)~\mathrm{d}s $. Then there holds~$p_k \in \mathcal{C}_0(I;\R^d)$
as well as
\begin{align*}
\langle p_k, v^k_i \rangle=(p_k(t^k_i),\mathbf{v}^k_i)_{\R^d}= \beta.
\end{align*}
In particular,~$\langle p_k, u'_k \rangle= \beta \sum^{N_k}_{i=1} \mu^k_i$ and~$\cnorm{p_k}\geq \beta$ if~$\mathcal{A}_k \neq \emptyset$. Moreover, if~$\cnorm{p_k}\leq \beta$ then~$u_k$ is a minimizer to~\eqref{def:bvprob}. In particular this holds if~$(\widehat{\mu},\widehat{v}^k) \in \mathcal{A}_k$ for some~$\widehat{\mu}>0$.
\end{prop}
\begin{proof}
By step 2. and 7., respectively, of Algorithm~\ref{alg:accgcg} we have~$\mu^k_i >0$. Moreover~$(\mu^k,c^k)$ is a minimizer to
\begin{align*}
\min_{(\mu,c)\in \R^{N_k}_+\times \R^d} \left \lbrack F\left( \mathcal{K} \left( \sum^{N_k}_{i=1} \mu_i v^k_i, c \right) \right)+ \beta \sum^{N_k}_{i=1} \mu_i \right \rbrack.
\end{align*}
It is readily verified that the first order necessary and sufficient optimality conditions for this problem imply
\begin{align} \label{eq:firstordersubhelp}
p_k(T)=0,~\langle p_k, v^k_i \rangle=(p_k(t^k_i),\mathbf{v}^k_i)_{\R^d}= \beta, \quad i=1,\dots,N_k
\end{align}
Consequently we get
\begin{align*}
\langle p_k, u'_k \rangle= \sum^{N_k}_{i=1} \mu^k_i \langle p_k, v^k_i \rangle=\beta \sum^{N_k}_{i=1} \mu^k_i
\end{align*}
as well as
\begin{align*}
\beta= \langle p_k, v^k_i \rangle \leq \cnorm{p_k}
\end{align*}
for every~$(\mu^k_i,v^k_i)\in \mathcal{A}_k$. Finally assume that~$\cnorm{p_k} \leq \beta$. If~$\cnorm{p_k}<\beta$ we note that~$\mathcal{A}_k=\emptyset$, i.e.~$u'_k=0$, and~$u_k$ satisfies the first order optimality conditions for~\eqref{def:bvprob}, see Theorem~\ref{thm:firstorderoptimality}. Hence, in this case,~$u_k$ is a minimizer to~\eqref{def:bvprob}. The same holds true if~$\cnorm{p_k}=\beta$ and~$u'_k=0$. Last let~$\cnorm{p_k}=\beta$ and~$u'_k \neq 0$ hold. Then~$\mathcal{A}_k \neq \emptyset$. Let~$(\mu^k_i,\mathbf{v}^k_i \delta_{t^k_i})\in \mathcal{A}_k$ be arbitrary. Summarizing the previous observations there holds
\begin{align*}
\beta=(p_k(t^k_i), \mathbf{v}^k_i)_{\R^d}=|p_k(t^k_i)|_{\R^d}= |p_k(t^k_i)|_{\R^d} |\mathbf{v}^k_i|_{\R^d}= \cnorm{p_k}
\end{align*}
and thus~$\mathbf{v}^k_i=p_k(t^k_i)/\beta$. Thus we conclude that~$\mathcal{A}_k$ is of the form
\begin{align*}
\mathcal{A}_k= \{\mu^k_i, p_k(t^k_i)/\beta\}^{N_k}_{i=1}
\end{align*}
with pairwise disjoint positions~$t^k_i$. Consequently,~$u_k=B\left( \sum^{N_k}_{i=1} \mu^k_i v^k_i, c^k \right)$ satisfies
\begin{align*}
\mnorm{u'_k}= \sum^{N_k}_{i=1} \mu^k_i |\mathbf{v}^k_i|_{\R^d}=\sum^{N_k}_{i=1} \mu^k_i.
\end{align*}
Together with~$\langle p_k,u'_k \rangle=\sum^{N_k}_{i=1} \mu^k_i$ we finish noting that~$u_k$ fulfils the sufficient first order optimality conditions from Theorem~\ref{thm:firstorderoptimality}. Finally, if~$(\widehat{\mu},\widehat{v}^k) \in \mathcal{A}_k$ for some~$\widehat{\mu}>0$ then we have
\begin{align*}
\beta= \langle p_k, \widehat{v}^k \rangle=\cnorm{p_k}
\end{align*}
and thus~$u_k$ is again a minimizer of~\eqref{def:bvprob} following the previous observations.
\end{proof}
\section{Convergence analysis} \label{sec:convergence}
This section addresses the convergence of Algorithm~\ref{alg:accgcg}. The presentation is split into two parts. In Section~\ref{subsec:worstcase} we provide the subsequential strict convergence of~$u_k$ towards minimizers of~\eqref{def:bvprob} as well as a first convergence result for the residuals
\begin{align*}
r_j(u_k) \coloneqq j(u_k)- \min_{u\in \BV} j(u).
\end{align*}
In the second part, Section~\ref{subsec:improved}, we prove that under additional structural assumptions on the optimal dual variable~$\bar{p}=\int^\cdot_0 K^* \nabla F(\bar{y})(s)~\mathrm{d}s$,~\eqref{def:bvprob} admits a unique minimizer~$\bar{u}$ and the iterates~$u_k$ generated by Algorithm~\ref{alg:accgcg} satisfy
\begin{align}
r_j(u_k)+\|u_k-\optu\|_{L^1}+ |\mnorm{u'_k}-\mnorm{\optu'}| \leq c \zeta^k
\end{align}
for some~$\zeta \in (0,1)$ and all~$k \in \N$ large enough.
\subsection{Global sublinear convergence} \label{subsec:worstcase}
In the following let
\begin{align*}
\mathcal{A}_k = \left\{\left( \mu^k_i, v^k_i \right)\right\}^{N_k}_{i=1},~u_k=B \left( \sum^{N_k}_{i=1} \mu^k_i v^k_i, c^k \right),~y_k=Ku_k ,~p_k(\cdot)= \int^\cdot_0 K^*(Ky_k)(s)~\mathrm{d}s
\end{align*}
denote the active set, iterate, observation and dual variable in iteration~$k$ of Algorithm~\ref{alg:accgcg}, respectively.
Since~$j$ is radially unbounded, see Assumption~\ref{ass:problem}, the norm of all elements in the sublevel set
\begin{align*}
E_{u_k}=\left \{\,u \in \BV\;|\;j(u)\leq F(Ku_k)+ \beta \sum^{N_k}_{i=1} \mu^k_i\,\right\}.
\end{align*}
is bounded by a constant~$M_k>0$. By construction there holds
\begin{align*}
F(Ku_{k+1})+ \beta \sum^{N_{k+1}}_{i=1} \mu^{k+1}_i \leq F(Ku_{k})+ \beta \sum^{N_{k}}_{i=1} \mu^{k}_i.
\end{align*}
Hence, w.l.o.g, we can assume that~$M_k$ is monotonically decreasing. For example, if~$F\geq 0$ on~$Y$, we can choose
\begin{align*}
M_k \coloneqq \left( F(Ku_k)+ \beta \sum^{N_k}_{i=1} \mu^k_i \right)/\beta.
\end{align*}
We require additional regularity assumptions on the loss functional~$F$.
\begin{assumption} \label{ass:assumponF}
The following two conditions hold:
\begin{itemize}
\item[\textbf{A1}] The gradient~$\nabla F$ is Lipschitz i.e. there is~$L>0$ such that
\begin{align*}
\ynorm{\nabla F(y_1)-\nabla F(y_2)} \leq L \, \ynorm{y_1-y_2} \quad \forall y_1,y_2 \in Y.
\end{align*}
\item[\textbf{A2}]
The functional~$F \colon Y \to \R $ is strongly convex around the optimal observation i.e. there exist a neighbourhood~$\mathcal{N}(\bar{y})$ of~$\bar{y}$ in $Y$ and~$\gamma_0>0$ with
\begin{align*}
F(y) \geq F(\bar{y})+ (\nabla F(\bar{y}),y-\bar{y})_Y+ \gamma_0 \|y-\bar{y}\|^2_Y \quad \forall y \in \mathcal{N}(\bar{y}).
\end{align*}
\end{itemize}
\end{assumption}
This is, e.g., fulfilled for the quadratic loss function~$F(\cdot)= (1/2) \ynorm{\cdot-y_d}^2$ with a target observation~$y_d \in Y$. Now define the auxiliary residual
\begin{align*}
\widehat{r}_j(u_k) \coloneqq F(Ku_k)+\beta \sum^{N_k}_{i=1} \mu^k_i- \min_{u\in \BV} j(u).
\end{align*}
Note that~$r_j(u_k) \leq \widehat{r}_j(u_k)$ holds due to
\begin{align*}
\mnorm{u'_k}= \left\|\sum^{N_k}_{i=1} \mu^k_i v^k_i \right\|_{\mathcal{M}} \leq \sum^{N_k}_{i=1} \mu^k_i
\end{align*}
using that~$\mnorm{v^k_i}=1$.
The following version of the classical descent lemma holds.
\begin{lemma} \label{lem:descent}
Let~$u_k\in \BV,~p_k \in \mathcal{C}_0(I;\R^d)$ and~$\widehat{v}^k \in \mathcal{M}(I;\R^d)$ be generated by Algorithm~\ref{alg:accgcg}. Then we have
\begin{align} \label{eq:descendlem}
\widehat{r}_j(u_{k+1})-\widehat{r}_j(u_{k}) \leq \min_{s\in[0,1]} \left \lbrack -s M_k \left( \cnorm{p_k}-\beta \right) + \frac{L s^2}{2} \ynorm{\mathcal{K}(u'_k-M_k\widehat{v}^k,0)}^2\right \rbrack
\end{align}
for all~$k\geq 1$.
\end{lemma}
\begin{proof}
For every~$s\in(0,1)$ define the auxiliary iterate~$u_{k,s}=B(u'_{k,s},c^k)$ where
\begin{align*}
u'_{k,s}= \mu^{k}_{s,N_k+1} \widehat{v}^k+ \sum^{N_k}_{i=1} \mu^{k}_{s,i}  v^k_i, \quad \text{where} \quad \mu^{k}_s \coloneqq \left((1-s)\mu^k, sM_0 \right) \in \R^{N_k+1}.
\end{align*}
Since~$u_{k+1}$ is constructed using a minimizing pair to~\eqref{def:subprobjumps} we have
\begin{align*}
\widehat{r}_j(u_{k+1})-\widehat{r}_j(u_{k}) \leq F(Ku_{k,s})-F(Ku_{k})+ \beta \left( \sum^{N_k +1}_{i=1} \mu^k_{s,i}- \sum^{N_k}_{i=1} \mu^k_i \right).
\end{align*}
By construction, the second term on the righthandside is equal to
\begin{align*}
\beta \left( \sum^{N_k +1}_{i=1} \mu^k_{s,i}- \sum^{N_k}_{i=1} \mu^k_i \right)= s \beta \left( M_k- \sum^{N_k}_{i=1} \mu^k_i\right).
\end{align*}
Using a Taylor's expansion of the first term~$F(Ku_{k,s})-F(Ku_{k})$ and utilizing the Lipschitz continuity of~$\nabla F$  yields
\begin{align*}
F(Ku_{k,s})-F(Ku_{k})\leq s\langle p_k,u'_k-M_k \widehat{v}^k \rangle+ \frac{L s^2}{2} \ynorm{\mathcal{K}(u'_k-\widehat{v}^k,0)}^2.
\end{align*}
Finally note that due to~$\widehat{v}^k=(p_k(\hat{t}_k)/\cnorm{p_k}) \delta_{\hat{t}_k}$ with~$|p_k(\hat{t}_k)|_{\R^d}=\cnorm{p_k}$ and Proposition~\ref{prop:optimalitysub} we have
\begin{align*}
s\langle p_k,u'_k-M_k \widehat{v}^k \rangle=s \left( \beta\sum^{N_k}_{i=1} \mu^k_i - M_k \cnorm{p_k}  \right).
\end{align*}
Summarizing all previous observations and minimizing w.r.t~$s\in [0,1]$ we arrive at the claimed inequality.
\end{proof}
Using Lemma~\ref{lem:descent} we prove the subsequential convergence of~$u_k$ towards minimizers of~\eqref{def:bvprob} as well as the sublinear convergence of ~$r_j(u_k)$.
\begin{theorem} \label{thm:slowconvalg}
Let~$u_k \in \BV$ and~$p_k \in \mathcal{C}_0(I;\R^d)$ be generated by Algorithm~\ref{alg:accgcg}. Then we have
\begin{align} \label{upperforres}
r_j(u_k) \leq \widehat{r}_j(u_k) \leq M_k(\cnorm{p_k}-\beta)
\end{align}
Moreover Algorithm~\ref{alg:accgcg} either terminates after finitely many steps with~$u_k$ a solution to~\eqref{def:bvprob} or we have
\begin{align} \label{rate}
r_j(u_k)\leq \widehat{r}_j(u_k)\leq \frac{\widehat{r}_j(u_1)}{1+qk} \quad \text{where} \quad q=\frac{1}{2} \min \left\{1, \frac{\widehat{r}_j(u_k)}{4 \|K\|^2_{ \text{BV},Y} M^2_0}\right\}
\end{align}
for all~$k \geq 1$.
In this case,~$u_k$ admits at least one strict accumulation point and each such point is a solution to~\eqref{def:bvprob}. Moreover we have~$K u_k \rightarrow \bar{y}$ in~$Y$ as well as~$p_k \rightarrow \bar{p}$ in~$\mathcal{C}_0(I;\R^d)$. If the minimizer~$\bar{u}$ to~\eqref{def:bvprob} is unique then~$u_k \rightharpoonup^s \bar{u}$ on the whole sequence.
\end{theorem}
\begin{proof}
Let~$\bar{u}$ denote an arbitrary minimizer of~\eqref{def:bvprob}. Since~$F$ is convex we estimate
\begin{align*}
\widehat{r}_j(u_k) \leq (-K^* \nabla F(Ku_k),\bar{u}-u_k)_Y+ \beta \left(\sum^{N_k}_{i=1} \mu^k_i-\mnorm{\bar{u}} \right)= \langle p_k, \bar{u}' \rangle- \beta \mnorm{\bar{u}}.
\end{align*}
Finally note that~$\bar{u} \in E_{u_k}$ and thus
\begin{align*}
\langle p_k, \bar{u}' \rangle- \beta \mnorm{\bar{u}} \leq \mnorm{\bar{u}}(\cnorm{p_k} -\beta) \leq M_k (\cnorm{p_k} -\beta)
\end{align*}
yielding~\eqref{upperforres}.

Now assume that Algorithm~\ref{alg:accgcg} does not converge after finitely many steps. Then~$\cnorm{p_k} \geq \beta$, see Proposition~\ref{prop:optimalitysub}, and ~$r_j(u_k)>0$ for all~$k$. Explicitly calculating the minimum in~\eqref{eq:descendlem}, using~\eqref{upperforres} and dividing by~$\widehat{r}_j(u_1)$ we obtain
\begin{align*}
\frac{\widehat{r}_j(u_{k+1})}{\widehat{r}_j(u_{1})} &\leq \frac{\widehat{r}_j(u_{k})}{\widehat{r}_j(u_{1})}- \frac{\widehat{r}_j(u_{1})}{2} \min \left\{\frac{1}{\widehat{r}_j(u_{k})}, \frac{1}{L \ynorm{\mathcal{K}(u'_k-M_k\widehat{v}^k,0)}^2}\right\} \frac{\widehat{r}_j(u_{k})}{\widehat{r}_j(u_{1})}
 \\ &\leq \frac{\widehat{r}_j(u_{k})}{\widehat{r}_j(u_{1})}- \frac{1}{2} \min \left\{1, \frac{\widehat{r}_j(u_k)}{4 \|K\|^2_{ \text{BV},Y} M^2_k}\right\} \frac{\widehat{r}_j(u_{k})}{\widehat{r}_j(u_{1})}
 \\ &\leq \frac{\widehat{r}_j(u_{k})}{\widehat{r}_j(u_{1})}- \frac{1}{2} \min \left\{1, \frac{\widehat{r}_j(u_k)}{4 \|K\|^2_{ \text{BV},Y} M^2_0}\right\} \frac{\widehat{r}_j(u_{k})}{\widehat{r}_j(u_{1})}.
\end{align*}
Invoking~\cite[Lemma 3.1]{Dunn}  yields~\eqref{rate}. Since~$j$ is radially unbounded, see Assumption~\ref{ass:assumponF}, and~$r_j(u_k) \rightarrow 0$, we conclude that~$u_k$ is bounded in~$\BV$. Thus it admits at least one weak* convergent subsequence, denoted by the same index, with limit~$\bar{u}\in \BV$ i.e.~$u_k \rightarrow \bar{u}$ in~$L^1(I;\R^d)$ and~$u'_k \rightharpoonup^* \bar{u}'$ in~$\Moc$. Since~$\BV \hookrightarrow^c L^2(I;\R^d)$ we also conclude~$y_k \rightarrow K\bar{u}$ in~$Y$ as well as
\begin{align*}
p_k \rightarrow \int^\cdot_0 K^* \nabla F(K\bar{u})(s)~\mathrm{d}s~\text{in}~\mathcal{C}_0(I;\R^d).
\end{align*}
Finally we note that~$j$ is weak* lower semicontinuous on~$\BV$. Consequently~$r_j(\bar{u})=0$ and~$\bar{u}$ is a minimizer to~\eqref{def:bvprob}. Finally, since~$F(Ku_k)\rightarrow F(K\bar{u})$, we also get~$\mnorm{u'_k} \rightarrow \mnorm{\bar{u}'}$ yielding the strict convergence of~$u_k$ towards~$\bar{u}$. Thus we have shown that any weak* accumulation point of $u_k$ is indeed a strict accumulation point and a minimizer of~\eqref{def:bvprob}. Recalling that the optimal observation~$\bar{y}$ as well as the optimal dual variable~$\bar{p}$ are unique we conclude~$y_k \rightarrow \bar{y}$ in~$Y$ and~$p_k \rightarrow \bar{p}$ in~$\mathcal{C}_0(I;\R^d)$ for the whole sequence. If~$\bar{u}$ is the unique minimizer of~\eqref{def:bvprob} then it is also the unique strict accumulation point of~$u_k$ and thus~$u^k \rightharpoonup^s \bar{u}$ on the whole sequence.
\end{proof}

If~$F$ is strongly convex around~$\bar{y}$, see Assumption~\ref{ass:assumponF}~$\mathbf{A2}$, then the convergence guarantee for the residual from Theorem~\ref{thm:slowconvalg} also carries over to the observations and dual variables.
\begin{prop}
\label{prop:estforstates}
Let~Assumption~\ref{ass:assumponF} hold. Then we have
\begin{align*}
\ynorm{y_k-\bar{y}}+ \cnorm{p_k-\bar{p}}+ |\cnorm{p_k}-\cnorm{\bar{p}}| \leq cr_j(u_k)^{1/2}
\end{align*}
for all~$k \in \N$ large enough.
\end{prop}
\begin{proof}
Let~$\mathcal{N}(\bar{y})$ denote the neighbourhood from Assumption~\ref{ass:assumponF}~$\mathbf{A2}$. Since~$y_k \rightarrow \bar{y}$ in~$Y$,see Theorem~\ref{thm:slowconvalg}, there holds~$y_k \in \mathcal{N}(\bar{y})$ for all~$k\in\N$ large enough. Consequently Assumption~\ref{ass:assumponF}~$\mathbf{A2}$ yields
\begin{align*}
r_j(u_k) \geq (\nabla F(\bar{y}),y_k-\bar{y})_Y + \beta(\mnorm{u'_k}-\mnorm{\bar{u}'})+ \gamma_0 \ynorm{y_k-\bar{y}}^2.
\end{align*}
Finally noting that
\begin{align*}
(\nabla F(K\bar{u}),y_k-\bar{y})_Y + \beta(\mnorm{u'_k}-\mnorm{\bar{u}'})= \langle \bar{p}, \bar{u}'-u'_k \rangle + \beta(\mnorm{u'_k}-\mnorm{\bar{u}'}) \geq 0,
\end{align*}
see Theorem~\ref{thm:firstorderoptimality}, we get
\begin{align*}
\ynorm{y_k-\bar{y}} \leq (1/\gamma_0)^{1/2}\, r_j(u_k)^{1/2}.
\end{align*}
The remaining estimates follow from
\begin{align*}
|\cnorm{p_k}-\cnorm{\bar{p}}| &\leq \cnorm{p_k-\bar{p}} \leq c \|K^*(\nabla F(y_k)-\nabla F(\bar{y}))\|_{L^2} \\ & \leq
 c \|K^*\|_{Y,L^2} \ynorm{\nabla F(y_k)-\nabla F(\bar{y})} \\& \leq
 c L\|K^*\|_{Y,L^2} \ynorm{y_k-\bar{y}}.
\end{align*}
\end{proof}
\subsection{Local linear convergence}
 \label{subsec:improved}
Next we prove that Algorithm~\ref{alg:accgcg} converges linearly provided that additional structural requirements on the optimal dual variable~$\bar{p}$ hold.  First we assume that~$\bar{p}$ only admits a finite number~$N$ of global extrema~$\{\bar{t}_i\}^N_{i=1}$. Together with a linear independence assumption on~$\{\bar{p}(\bar{t}_i)\}^N_{i=1}$ this ensures the existence of a unique, piecewise constant minimizer to~\eqref{def:bvprob}.
\begin{assumption} \label{ass:nondegeneratesource1}
Recall the definition of the optimal dual variable~$\bar{p}= \int^\cdot_0 K^*\nabla F(\bar{y})~\mathrm{d}s$. Assume that there is~$N\in \N$ and~$\{\bar{t}_i\}^N_{i=1}\subset I$ with
\begin{align} \label{ass:isolpoints}
 \{\bar{t}_i\}^N_{i=1}=\left\{\,t \in I\;|\;|\bar{p}(t)|_{\R^d}=\cnorm{\bar{p}}=\beta\, \right\}.
\end{align}
Moreover let~$\{e_i\}^N_{i=1} \subset \R^d$ denote the canonical basis of~$\R^d$. The set
\begin{align} \label{eq:linindass}
\{K\bar{p}(\bar{t}_i) \chi_{\bar{t}_i}\}^N_{i=1} \cup \{Ke_i\chi_I\}^d_{i=1} \subset Y
\end{align}
is linearly independent.
\end{assumption}
\begin{coroll} \label{coroll:uniqueness}
Let Assumption~\ref{ass:nondegeneratesource1} hold. Then the minimizer~$\bar{u}=B(\bar{u}',a_{\bar{u}})$ to~\eqref{def:bvprob} is unique and~$\bar{u}'$ is given by
\begin{align*}
\bar{u}'= \sum^N_{i=1} \bar{\mu}_i \bar{v}_i= \sum^N_{i=1} \bar{\mu}_i \bar{\mathbf{v}}_i \delta_{\bar{t}^i_0} \quad \text{where} \quad \bar{\mu}_i \geq 0,~\bar{\mathbf{v}}_i= \frac{\bar{p}(\bar{t}_i)}{\beta}
\end{align*}
for all~$i=1,\dots,N$.
\end{coroll}
\begin{proof}
Introduce the linear and continuous operator~$\widehat{\mathbf{K}} \colon \R^N \times \R^d \to Y$ by
\begin{align*}
\widehat{\mathbf{K}}(\mu,C)= K \left(C \chi_0  \right)+\sum^N_{i=1} K \left((\bar{p}(\bar{t}_i)/\beta) \chi_{\bar{t}_i}\right) \quad \forall \mu \in \R^N,~C \in \R^d.
\end{align*}
Then~$\widehat{\mathbf{K}}$ is injective according to~\eqref{eq:linindass}.
According to Corollary~\ref{coroll:suppcond} and~\eqref{ass:isolpoints} every minimizer~$\optu$ of~\eqref{def:bvprob} is of the form
\begin{align*}
\bar{u}= B \left( \sum_{i=1} (\bar{p}(\bar{t}_i)/\beta) \delta_{\bar{t}_i}, a_{\bar{u}} \right)= \bar{C} \chi_0 + \sum^N_{i=1} \bar{\mu}_i  \chi_{\bar{t}_i}
\end{align*}
where~$\bar{\mu} \in \R^N_+$ and~$\bar{C}$ is implicitly given by
\begin{align*}
\bar{C}=a_{\bar{u}}-\frac{1}{T} \sum^N_{i=1}\bar{\mu}_i \, (\bar{p}(\bar{t}_i)/\beta) )\,(T-\bar{t}_i),
\end{align*}
see the definition of the operator~$B$,~\eqref{def:Boperator}, and its inverse~$B^{-1}$, respectively. Due to the optimality of~$\bar{u}$ for~\eqref{def:bvprob} we readily verify that~$(\bar{\mu},\bar{C})$ is a minimizing pair for
\begin{align} \label{def:uniquenessprob}
\min_{\mu \in \R^N_+, C \in \R^d} \left \lbrack F(\widehat{\mathbf{K}}(\mu,C))+\beta \sum^N_{i=1} \mu_i \right \rbrack.
\end{align}
The proof is finished noting that~\eqref{def:uniquenessprob} admits a unique minimizer since~$F \circ \widehat{\mathbf{K}}$ is strictly convex.
\end{proof}
According to Assumption~\ref{ass:nondegeneratesource1} and the continuity of~$|\bar{p}|_{\R^d}$ there is~$\sigma>0$ as well as a radius~$R>0$ such that the intervals~$(\bar{t}_i-R,\bar{t}_i+R) \subset I$,~$i=1,\dots,N$, are pairwise disjoint and
\begin{align} \label{eq:isolofvals}
|\bar{p}(t)|_{\R^d} \leq \beta -\sigma \quad \forall t \in  \bar{I} \setminus \bigcup^N_{i=1} (\bar{t}_i-R,\bar{t}_i+R).
\end{align}
Now we impose a final set of assumptions which requires the positivity of~$\bar{\mu}_i$ as well as the quadratic growth of~$|\bar{p}|_{\R^d}$ around its global maximizers. From the perspective of optimization, this first condition corresponds to a~\textit{strict complementarity condition} and the second one is equivalent to a~\textit{second-order-sufficient-condition (SSC)} for~$\bar{t}_i$.
\begin{assumption} \label{ass:nondegeneratesource2}
For all~$i=1,\dots,N$, there holds~$\bar{\mu}_i >0$ as well as
\begin{align*}
\theta_0|t-\bar{t}_i|^2 \leq \beta -|\bar{p}(t)|_{\R^d} \quad \forall t\in (\bar{t}_i-R,\bar{t}_i+R)
\end{align*}
where~$R>0$ denotes the radius from~\eqref{eq:isolofvals}. Moreover $K^* \in \mathcal{L}(Y; L^{\infty}(I;\R^d))$.
\end{assumption}
\begin{remark} \label{rem:quadgrowth}
Define the scalar-valued function~$\bar{P}(t)=|\bar{p}(t)|_{\R^d}$ and assume that~$\bar{p}\in \mathcal{C}^2(I;\R^d)$. Then it is readily verified that~$\bar{P}$ is also at least two times continuously differentiable on~$(\bar{t}_i-R,\bar{t}+R)$ if~$R>0$ is chosen small enough. In particular this implies~$\bar{P}(\bar{t}_i)=\beta,~\bar{P}'(\bar{t}_i)=0$ and~$\bar{P}''(\bar{t}_i)\leq 0$. Thus, by potentially choosing~$R>0$ even smaller as well as Taylor approximation of~$\bar{P}$ we arrive at
\begin{align*}
\bar{P}(t)= \leq \beta - \frac{|\bar{P}''(\bar{t})|}{4} |t-\bar{t}_i|^2
\end{align*}
for all~$t \in (\bar{t}_i-R,\bar{t}_i+R)$. Hence the quadratic growth condition of Assumption~\ref{ass:nondegeneratesource2} is fulfilled if~$\bar{P}''(\bar{t}_i)\neq 0$,~$i=1,\dots,N$.
\begin{align*}
\end{align*}
\end{remark}
The following quadratic growth behaviour of the linear functional induced by~$\bar{p}$ is a direct consequence.
\begin{lemma} \label{lem:quadgrowthlinearfunc}
Let Assumption~\ref{ass:nondegeneratesource2} hold. Then there is~$\gamma_1 >0$ such that
\begin{align*}
\gamma_1 \left( |t-\bar{t}_i|^2+|\mathbf{v}-\bar{\mathbf{v}}_i|^2_{\R^d} \right) \leq \beta -\langle \bar{p}, \mathbf{v} \delta_t \rangle \quad \forall t\in (\bar{t}_i-R,\bar{t}_i+R),~|\mathbf{v}|_{\R^d}=1
\end{align*}
and all~$i=1,\dots,N$.
\end{lemma}
\begin{proof}
Fix~$i=1,\dots,N$ and~$ t\in (\bar{t}_i-R,\bar{t}_i+R)$ as well as~$\mathbf{v}\in\R^d$ with~$|\mathbf{v}|_{\R^d}=1$.
From Assumption~\ref{ass:nondegeneratesource2} and~$|\mathbf{v}|_{\R^d}=1$ we immediately get
\begin{align*}
\beta -\langle \bar{p}, \mathbf{v} \delta_t \rangle \geq \beta-|\bar{p}(t)|_{\R^d} \geq \theta_0 |t-\bar{t}_i|.
\end{align*}
Second we estimate
\begin{align*}
\beta\left(1-\langle \bar{p}/\beta, \mathbf{v} \delta_t  \rangle\right)= \beta\left(1-( \bar{p}(t)/\beta, \mathbf{v} )_{\R^d}\right) \geq \frac{\beta}{2} |\bar{p}(t)/\beta-\mathbf{v}|^2_{\R^d}.
\end{align*}
using~$|\bar{p}(t)/\beta|_{\R^d}\leq 1$.
Finally we have
\begin{align} \label{eq:lpischitzbarp}
|\bar{p}(t)-\bar{p}(\bar{t}_i)|_{\R^d} \leq |t-\bar{t}_i| \|K^*\nabla F(K\bar{u})\|_{L^\infty(I;\R^n)}.
\end{align}
The claimed statement now follows from noting that
\begin{align*}
 |t-\bar{t}_i|^2+|\mathbf{v}-\bar{\mathbf{v}}_i|^2_{\R^d}  \leq  |t-\bar{t}_i|^2+ 2\left( |\mathbf{v}-\bar{p}(t)/\beta|^2_{\R^d}+ |(\bar{p}(t)-\bar{p}(\bar{t}_i))/\beta|^2_{\R^d} \right)
\end{align*}
where~$\bar{\mathbf{v}}_i= \bar{p}(\bar{t}_i)/\beta$ is used in the first inequality.
\end{proof}
Moreover we deduce the following Lipschitz property of~$\mathcal{K}$.
\begin{lemma} \label{lem:lipofcalK}
There holds
\begin{align*}
\ynorm{\mathcal{K}(\mathbf{v}_1 \delta_{t_1}-\mathbf{v}_2 \delta_{t_2},0)} \leq c \left( |t_1-t_2|+|\mathbf{v}_1-\mathbf{v}_2|_{\R^d} \right)
\end{align*}
for all~$t_1,t_2 \in I$,~$\mathbf{v}_1,\mathbf{v}_2 \in \R^d$,~$|\mathbf{v}_1|_{\R^d}=|\mathbf{v}_2|_{\R^d}=1$.
\end{lemma}
\begin{proof}
Using the additional regularity of~$K^*$ from Assumption~\ref{ass:nondegeneratesource2} we get
\begin{align*}
\ynorm{\mathcal{K}(\mathbf{v}_1 \delta_{t_1}-\mathbf{v}_2 \delta_{t_2},0)}&= \sup_{\ynorm{y}=1} (\mathcal{K}(\mathbf{v}_1 \delta_{t_1}-\mathbf{v}_2 \delta_{t_2},0),y)_Y\\&= \sup_{\ynorm{y}=1} (B(\mathbf{v}_1 \delta_{t_1}-\mathbf{v}_2 \delta_{t_2},0),K^*y)_{L^2}
\\& \leq \|K^*y\|_{L^\infty} \|B(\mathbf{v}_1 \delta_{t_1}-\mathbf{v}_2 \delta_{t_2},0)\|_{L^1} \\& \leq
\|K^*\|_{Y,L^\infty} \|B(\mathbf{v}_1 \delta_{t_1}-\mathbf{v}_2 \delta_{t_2},0)\|_{L^1}.
\end{align*}
Now recall that
\begin{align*}
B(\mathbf{v}_i \delta_{t_i},0)= \mathbf{v}_i \chi_{t_i}- \frac{1}{T} \mathbf{v}_i (T-t_i),
\end{align*}
$i=1,2$, and thus
\begin{align*}
\|B(\mathbf{v}_1 \delta_{t_1}-\mathbf{v}_2 \delta_{t_2},0)\|_{L^1(I)} \leq \|\mathbf{v}_1 \chi_{t_1}-\mathbf{v}_2 \chi_{t_2}\|_{L^1} +  |\mathbf{v}_1 (T-t_1)-\mathbf{v}_2 (T-t_2)|_{\R^d}.
\end{align*}
The proof is finished noting that
\begin{align} \label{eq:estforindic}
\|\mathbf{v}_1 \chi_{t_1}-\mathbf{v}_2 \chi_{t_2}\|_{L^1} &\leq T |\mathbf{v}_1- \mathbf{v}_2|_{\R^d}+ |\mathbf{v}_1|_{\R^d}\| \chi_{t_1}- \chi_{t_2}\|_{L^1}\\ & \leq |t_1-t_2|+ T |\mathbf{v}_1-\mathbf{v}_2|_{\R^d}
\end{align}
as well as
\begin{align*}
|\mathbf{v}_1 (T-t_1)-\mathbf{v}_2 (T-t_2)|_{\R^d} & \leq |\mathbf{v}_1|_{\R^d}|t_1-t_2|+ |T-t_2||\mathbf{v}_1-\mathbf{v}_2|_{\R^d} \\ &\leq |t_1-t_2|+ T |\mathbf{v}_1-\mathbf{v}_2|_{\R^d}.
\end{align*}
\end{proof}
\subsubsection*{Sketch of the proof}
The following theorem summarizes the main results of the following sections.
\begin{theorem}
Let~$u_k$ be generated by Algorithm~\ref{alg:accgcg} and let Assumption~\ref{ass:problem}-\ref{ass:nondegeneratesource2} hold. Then Algorithm~\ref{alg:accgcg} either terminates after finitely many steps with~$u_k= \bar{u}$ or there is~$\zeta\in (0,1)$ such that
\begin{align*}
r_j(u_k)+\|u_k-\optu\|_{L^1}+ |\mnorm{u'_k}-\mnorm{\optu'}| \leq c \zeta^k
\end{align*}
for all~$k\in\N$ large enough.
\end{theorem}
Since the proof of this improved convergence behaviour is rather technical we give a short outline before going into detail. Utilizing the strict convergence of~$u_k$ towards~$\bar{u}$ as well as the isolation of the global extrema of~$\bar{p}$ we conclude that the iterate~$u_k$ only jumps in the vicinity of~$\{\bar{t}_i\}^N_{i=1}$. More in detail, for sufficiently large~$k$, these observations yield a partition of~$\{1,\dots,N_k\}$ into nonempty, pairwise disjoint sets~$A^i_k$,~$i=1,\dots,N$, such that
\begin{align*}
 (\mu^k_i, \mathbf{v}^k_j \delta_{t^k_j})\in \mathcal{A}_k,~j \in A^i_k \Rightarrow~t^k_j \in (\bar{t}_i-R, \bar{t}_i +R).
\end{align*}
Moreover the "closedness" of the jumps~$v^k_j$,~$j\in A^i_k$, and the optimal one~$\bar{v}_i$,  i.e. the distance between the positions~$t^k_j$ and~$\bar{t}_i$ as well as the misfit between the associated directions~$\mathbf{v}^k_i-\bar{\mathbf{v}}_i$, can be quantified in terms of the auxiliary residual~$\widehat{r}_j(u_k)$, see Lemma~\ref{lem:quantit}. Similarly, in Proposition~\ref{prop:localmaximizers}, we show that the new candidate jump~$\widehat{v}^k$, see step 5. in Algorithm~\ref{alg:accgcg}, lies in the vicinity of some~$\bar{t}_{\widehat{\imath}} \in \{\bar{t}_i\}^N_{i=1}$. Finally, as in the proof of Lemma~\ref{lem:descent}, we then rely on an auxiliary iterate~$\widehat{u}_{k,s}=B(\widehat{u}'_{k,s},c^k)$,~$s\in(0,1)$, where
\begin{align*}
\widehat{u}'_{k,s}&= (1-s)\sum_{j \in A^{\widehat{\imath}}_k} \mu^k_j v^k_j+ s\left( \sum_{j \in A^{\widehat{\imath}}_k} \mu^k_j \right)\frac{p_k(\hat{t}_k)}{\cnorm{p_k}} \delta_{\hat{t}_k}+ \sum^N_{\substack{i=1, \\i \neq \widehat{\imath}}} \sum_{j \in A^i_k} \mu^k_j v^k_j
\end{align*}
The descent properties of this auxiliary iterate are then exploited in Lemma~\ref{lem:fundamentallemma} to prove an improved version of the descent lemma, Lemma~\ref{lem:descent}, which finally yields the linear convergence of~$r_j(u_k)$. The linear convergence of~$u_k$ w.r.t to the strict topology is then concluded as a by-product, see Lemmas~\ref{lem:estofnorms} and~\ref{lem:estofl1}.
\begin{remark}
To finish this section let us briefly compare~$\widehat{u}_{k,s}$ with the auxiliary iterate~$u_{k,s}=B(u'_{k,s},c^k)$, where
\begin{align*}
u'_{k,s}= sM_0 \widehat{v}_k+ (1-s) \sum^{N_k}_{i=1} \sum_{j \in A^i_k} \mu^k_iv^k_i=(1-s)u'_k+ s M_0  \widehat{v}_k
\end{align*}
which is used in the proof of Lemma~\ref{lem:descent}. Loosely speaking, to obtain~$u'_{k,s}$ we take "mass" from~\textit{all} Dirac Delta functionals in~$u'_k$, i.e. the height of all jumps in the iterate is decreased, and move it to the new candidate jump~$\widehat{v}^k$. In contrast, the construction of~$\widehat{u}_{k,s}$ can be viewed as a~\textit{local} update of~$u_k$ since mass is only taken away from those jumps~$u^k_j$ supported in~$(\bar{t}_{\widehat{\imath}}-R, \bar{t}_{\widehat{\imath}}+R)$. On the complement,~$I \setminus(\bar{t}_{\widehat{\imath}}-R, \bar{t}_{\widehat{\imath}}+R)$, we have~$\widehat{u}_{k,s}=u_k$. This allows for a refined analysis of the descent achieved by Algorithm~\ref{alg:accgcg} in each iteration.
\end{remark}
\subsubsection*{Linear convergence of the residual}
For the sake of readability we tacitly assume that Algorithm~\ref{alg:accgcg} does not converge after finitely many steps. The following proposition summarizes some immediate consequences of this assumption.
\begin{prop}
Assume that Algorithm~\ref{alg:accgcg} does not terminate after finitely many steps. Then there holds~$u_k \rightharpoonup^s  \bar{u}$,~$\cnorm{p_k} \geq \beta$ and~$\mathcal{A}_k \neq \emptyset$ for all~$k\in\N$ large enough.
\end{prop}
\begin{proof}
Since the minimizer to~\eqref{def:bvprob} is unique, see Corollary~\ref{coroll:uniqueness}, we get~$u_k \rightharpoonup^s \bar{u}$ from Theorem~\ref{thm:slowconvalg}. In particular, this implies~$u'_k \rightharpoonup^* \bar{u}'$ in~$\Moc$ and thus~$u'_k \neq 0$ for~$k$ large enough. This also yields~$\mathcal{A}_k \neq \emptyset$ and~$\cnorm{p_k}\geq \beta$, see Proposition~\ref{prop:optimalitysub}.
\end{proof}
Now we use the isolation of the global extrema of~$\bar{p}$, see~\eqref{eq:isolofvals}, as well as the uniform convergence of~$p_k$ from Proposition~\ref{prop:estforstates} to conclude that~$p_k$ is small outside of the intervals~$ (\bar{t}_i-R,\bar{t}_i+R)$.
\begin{coroll} \label{coroll:isolvalit}
Let~$\sigma>0$ and~$R>0$ as in~\eqref{eq:isolofvals} be given. Moreover let~$p_k$ be generated by Algorithm~\ref{alg:accgcg}. For all~$k\in\N$ large enough we have
\begin{align*}
|p_k(t)|_{\R^d} \leq \beta-\frac{\sigma}{2} \quad \forall t\in \bar{I} \setminus \bigcup^N_{i=1} (\bar{t}_i-R,\bar{t}_i+R).
\end{align*}
\end{coroll}
\begin{proof}
Choose an arbitrary but fixed~$t\in \bar{I} \setminus \bigcup^N_{i=1} (\bar{t}_i-R,\bar{t}_i+R)$. We estimate
\begin{align*}
|p_k(t)|_{\R^d} \leq |\bar{p}(t)|_{\R^d}+ ||p_k(t)|_{\R^d}-|\bar{p}(t)|_{\R^d}| \leq \beta-\sigma+ \cnorm{p_k-\bar{p}} \leq \beta- \frac{\sigma}{2}
\end{align*}
for all~$k\in\N$ large enough. Here we use~\eqref{eq:isolofvals} in the second inequality and the uniform convergence of~$p_k$, see Proposition~\ref{prop:estforstates}, in the last one.
\end{proof}
Using this estimate we prove that the iterate~$u_k$ solely jumps in the vicinity of the optimal jump positions~$\bar{t}_i$.
\begin{prop} \label{prop:localofoit}
Denote by
\begin{align*}
\mathcal{A}_k= \left\{ (\mu^k_i, v^k_i)\right\}^{N_k}_{i=1}=\left\{ (\mu^k_i, \mathbf{v}^k_i\delta_{t^k_i})\right\}^{N_k}_{i=1}
\end{align*}
the sequence of active sets generated by Algorithm~\ref{alg:accgcg}.
For all~$k\in\N$ large enough there exist pairwise disjoint index sets~$A^i_k$ with~$\bigcup^N_{i=1} A^i_k=\{1,\dots,N_k\}$ and~$t^k_j\in (\bar{t}_i-R, \bar{t}_i+R),~j \in A^i_k$.
\end{prop}
\begin{proof}
Let~$(\mu^k_j, v^k_j)=(\mu^k_j, \mathbf{v}^k_j\delta_{t^k_j})\in \mathcal{A}_k$ be arbitrary.
Utilizing the first order optimality condition for the subproblem~\eqref{def:subprobjumps}, see Proposition~\ref{prop:optimalitysub}, we have
\begin{align*}
\beta= (p_k(t^k_j),\mathbf{v}^k_j)_{\R^d} \leq |p_k(t^k_j)|_{\R^d}.
\end{align*}
Thus, together with Corollary~\ref{coroll:isolvalit}, we conclude~$t^k_j \in (\bar{t}_i-R,\bar{t}_i+R)$ for exactly one~$i \in\{1,\dots,N\}$. The existence of the index sets~$A^i_k$ is now imminent.
\end{proof}
Next we prove that the sets~$A^i_k$ are nonempty for large~$k\in\N$. This means that each optimal jump~$\bar{v}_i$ is approximated by at least one jump in the iterate~$u_k$. Moreover the "lumped" height~$\sum_{j \in A^i_k} \mu^k_j$ of all jumps~$v^k_j,~j \in A^i_k$, converges to the optimal jump height~$\bar{\mu}_i$. For this purpose define the restricted measures
\begin{align} \label{def:restricmeasures}
U'_{k,i} \coloneqq \sum_{j \in A^i_k} \mu^k_j v^k_j.
\end{align}
\begin{lemma} \label{lem:convlocalized}
Let~$U'_{k,i}$ be defined as in~\eqref{def:restricmeasures}. Then there holds
\begin{align*}
U'_{k,i} \rightharpoonup^* \bar{\mu}_i \bar{v}_i \delta_{\bar{t}_i},~\sum_{j \in A^i_k} \mu^k_j \rightarrow \bar{\mu}_i.
\end{align*}
In particular this implies~$A^i_k \neq \emptyset$ for all~$k\in\N$ and~$\sum^{N_k}_{i=1} \mu^k_i \rightarrow \mnorm{\bar{u}'}$.
\end{lemma}
\begin{proof}
Let~$i=1,\dots,N$ be arbitrary but fixed
and let~$\chi \in \mathcal{C}_0(I)$ be such that~$\chi(t)=1$,~$t \in (\bar{t}_i-R,\bar{t}_i+R)$, as well as $\chi(t)=0$,~$t \in (\bar{t}_j-R,\bar{t}_j+R)$,~$j\neq i$. Moreover denote by~$\varphi \in \mathcal{C}_0(I;\R^d)$ an arbitrary test function. Then we have~$\chi\varphi \in \mathcal{C}_0(I;\R^d)$ and thus
\begin{align*}
\langle \varphi, U'_{k,i} \rangle=\langle \chi \varphi, u'_k\rangle \rightarrow \langle \chi \varphi, \bar{u} \rangle=\langle \varphi, \bar{\mu}_i \bar{u}_i\rangle
\end{align*}
due to~$u'_k \rightharpoonup^* \bar{u}$, see Theorem~\ref{thm:slowconvalg}.
Consequently~$U'_{k,i} \rightharpoonup^* \bar{\mu}_i \bar{u}_i$. Similarly we conclude
\begin{align*}
\beta \sum_{j \in A^i_k} \mu^k_j =\langle \chi p_k, u'_k \rangle=\langle \chi \bar{p}, \bar{u}' \rangle= \beta \bar{\mu}_i
\end{align*}
using the first order optimality conditions for~$u_k$ and~$\bar{u}$, see Proposition~\ref{prop:optimalitysub} and Theorem~\ref{thm:firstorderoptimality}, respectively, as well as~$p_k \rightarrow \bar{p}$ in~$\mathcal{C}_0(I;\R^d)$, see Proposition~\ref{prop:estforstates}. Thus~$A^i_k \neq \emptyset$ for all~$k\in\N$ large enough.
The last statement now follows due to
\begin{align*}
\sum^{N_k}_{i=1} \mu^k_i= \sum^N_{i=1} \sum_{j \in A^i_k} \mu^k_j.
\end{align*}
\end{proof}
Up to now we have only given qualitative statements on the approximation of~$\bar{v}_i$ by jumps~$v^k_j$ of the iterate~$u_k$. In order to improve on the convergence result of Theorem~\ref{thm:slowconvalg} we also need a quantitative estimate for this observation. For this purpose we recall that both,~$\bar{v}_i$ and~$v^k_j$, are vector-valued Dirac Delta functionals. Thus, a suitable way to compare these jumps is given in terms of the differences~$t^k_j-\bar{t}_i$ and~$\mathbf{v}^k_j-\bar{\mathbf{v}}_i$ of jump positions and directions, respectively. This can be quantified using the quadratic growth behaviour of~$\bar{p}$ from Lemma~\ref{lem:quadgrowthlinearfunc}.
\begin{lemma} \label{lem:quantit}
There holds
\begin{align} \label{eq:l2likeest}
\sum^{N}_{i=1} \sum_{j \in A^i_k} \mu^k_j \left( |t^k_j-\bar{t}_i|+|\mathbf{v}^k_j-\bar{\mathbf{v}}_i|_{\R^d} \right) \leq c \sqrt{\widehat{r}_j(u_k)}
\end{align}
for all~$k\in\N$ large enough.
\end{lemma}
\begin{proof}
Let~$\gamma_1$ denote the constant from Lemma~\ref{lem:quadgrowthlinearfunc}. Applying Jensen's inequality yields
\begin{align*}
\frac{\gamma_1}{2 \sum^{N_k}_{i=1} \mu^k_i}&\left( \sum^{N}_{i=1} \sum_{j \in A^i_k} \mu^k_j \left( |t^k_j-\bar{t}_i|+|\mathbf{v}^k_j-\bar{\mathbf{v}}_i|_{\R^d} \right) \right)^2
\\& \leq \frac{\gamma_1}{2} \sum^{N}_{i=1} \sum_{j \in A^i_k} \mu^k_j \left( |t^k_j-\bar{t}_i|+|\mathbf{v}^k_j-\bar{\mathbf{v}}_i|_{\R^d} \right)^2 \\
& \leq {\gamma_1} \sum^{N}_{i=1} \sum_{j \in A^i_k} \mu^k_j \left( |t^k_j-\bar{t}_i|^2+|\mathbf{v}^k_j-\bar{\mathbf{v}}_i|^2_{\R^d} \right)
 \\& \leq \sum^{N}_{i=1} \sum_{j \in A^i_k} \mu^k_j \left( \beta- \langle \bar{p},v^k_j \rangle \right)=\beta \sum^{N_k}_{i=1} \mu^k_i-\langle\bar{p},u'_k \rangle.
\end{align*}
Moreover, due to the convexity of~$F$ we estimate
\begin{align*}
\widehat{r}_j(u_k)&= F(Ku_k)+ \beta \sum^{N_k}_{i=1} \mu^k_i- F(K\bar{u})-\beta \mnorm{\bar{u}}\\& \geq \beta \sum^{N_k}_{i=1} \mu^k_i- \beta \mnorm{\bar{u}}+ (\nabla F(K\bar{u}),K\bar{u_k}-K\bar{u})_Y.
\end{align*}
Now we rewrite
\begin{align*}
(\nabla F(K\bar{u}),K\bar{u_k}-K\bar{u})_Y- \beta \mnorm{\bar{u}}= \langle \bar{p},\bar{u}'-u'_k \rangle - \beta \mnorm{\bar{u}}=-\langle \bar{p},u'_k\rangle.
\end{align*}
using the first order optimality conditions for~$\bar{u}$, see Theorem~\ref{thm:firstorderoptimality}.
Summarizing all previous observations we arrive at
\begin{align*}
\left( \sum^{N}_{i=1} \sum_{j \in A^i_k} \mu^k_j \left( |t^k_j-\bar{t}_i|+|\mathbf{v}^k_j-\bar{u}_i|_{\R^d} \right) \right)^2 \leq \frac{2 \sum^{N_k}_{i=1} \mu^k_i}{\gamma_1}\, \widehat{r}_j(u_k) \leq \frac{4 \mnorm{\bar{u}}}{\gamma_1}\, \widehat{r}_j(u_k).
\end{align*}
Taking the square root on both sides of the inequality yields the claimed statement.
\end{proof}
A similar estimate holds for the new candidate jump~$\widehat{v}^k$ computed in step 3. of Algorithm~\ref{alg:accgcg}.
\begin{prop} \label{prop:localmaximizers}
Let~$\widehat{v}^k= (p_k(\hat{t}_k)/\cnorm{p_k})\delta_{\bar{t}_k}$ with~$|p_k(\hat{t}_k)|_{\R^d}=\cnorm{p_k}$ be given. For all~$k\in\N$ large enough there is a k-dependent index~$\widehat{\imath}\in \{1, \dots,N\}$ such that~$\hat{t}_k \in A^{\widehat{\imath}}_k$ and
\begin{align} \label{eq:estonpoints}
|\hat{t}_k-\bar{t}_{\widehat{\imath}}|+|p_k(\hat{t}_k)/\cnorm{p_k}-\bar{\mathbf{v}}_{\widehat{\imath}}|_{\R^d} \leq c \sqrt{r_j(u_k)}
\end{align}
\end{prop}
\begin{proof}
According to Proposition~\ref{prop:estforstates} there holds~$\cnorm{p_k} \rightarrow \beta$. Thus we conclude~$\hat{t}_k \in (\bar{t}_{\widehat{\imath}}-R,\bar{t}_{\widehat{\imath}}+R) $ for some~$\widehat{\imath}\in \{1,\dots,N\}$ from Corollary~\ref{coroll:isolvalit}. Applying Lemma~\ref{lem:quadgrowthlinearfunc} we get
\begin{align*}
\frac{\gamma_1}{2}\left(|\hat{t}_k-\bar{t}_{\widehat{\imath}}|+|p_k(\hat{t}_k)/\cnorm{p_k}-\bar{\mathbf{v}}_{\widehat{\imath}}|_{\R^d} \right)^2 &\leq \gamma_1 \left(|\hat{t}_k-\bar{t}_{\widehat{\imath}}|^2+|p_k(\hat{t}_k)/\cnorm{p_k}-\bar{\mathbf{v}}_{\widehat{\imath}}|^2_{\R^d} \right) \\& \leq \beta- \langle \bar{p}, \widehat{v}^k \rangle.
\end{align*}
Next note that
\begin{align*}
\beta- \langle \bar{p}, \widehat{v}^k \rangle= \langle \bar{p}, \bar{u}'_{\widehat{\imath}}-\widehat{v}^k \rangle \leq \langle \bar{p}-p_k, \bar{u}'_{\widehat{\imath}}-\widehat{v}^k \rangle=(\nabla F(\bar{y})-\nabla F(y_k), \mathcal{K}(\bar{u}'_{\widehat{\imath}}-\widehat{v}^k,0))_Y
\end{align*}
since
\begin{align*}
\cnorm{p_k}= \langle p_k, \widehat{v}^k \rangle \geq \langle p_k, \bar{u}'_{\widehat{\imath}}\rangle.
\end{align*}
Utilizing Proposition~\ref{prop:estforstates} and Lemma~\ref{lem:lipofcalK} we finally arrive at
\begin{align*}
\frac{\gamma_1}{2}\left(|\hat{t}_k-\bar{t}_{\widehat{\imath}}|+|p_k(\hat{t}_k)/\cnorm{p_k}-\bar{\mathbf{v}}_{\widehat{\imath}}|_{\R^d} \right)^2  &\leq \ynorm{\nabla F(\bar{y})-\nabla F(y_k)} \ynorm{\mathcal{K}(\bar{u}_i-\widehat{v}^k,0)} \\
& \leq c \sqrt{r_j(u_k)} \left(|\hat{t}_k-\bar{t}_{\widehat{\imath}}|+|p_k(\hat{t}_k)/\cnorm{p_k}-\bar{\mathbf{v}}_{\widehat{\imath}}|_{\R^d} \right).
\end{align*}
\end{proof}
Now fix~$k\in \N$ large enough and let~$\widehat{\imath}\in \{1,\dots,N\}$ be the index from Proposition~\ref{prop:localmaximizers}. Further recall the index sets~$A^i_k$,~$i=1,\dots,N$, from Proposition~\ref{prop:localofoit}. For every~$s\in [0,1]$ define the locally lumped measure
\begin{align*}
\widehat{u}'_{k,s}&= (1-s)\sum_{j \in A^{\widehat{\imath}}_k} \mu^k_j v^k_j+ s\left( \sum_{j \in A^{\widehat{\imath}}_k} \mu^k_j \right)\frac{p_k(\hat{t}_k)}{\cnorm{p_k}} \delta_{\hat{t}_k}+ \sum^N_{\substack{i=1, \\i \neq \widehat{\imath}}} \sum_{j \in A^i_k} \mu^k_j v^k_j =
\widehat{\mu}^k_{s,N_k +1}\widehat{v}^k+\sum^{N_k }_{j=1} \widehat{\mu}^k_{s,j} v^k_j
\end{align*}
where~$\widehat{\mu}^k_s \in \R^{N_k +1}$ is defined as
\begin{align} \label{def:optcoeffs}
\widehat{\mu}^k_{s,j}= \mu^k_j,~\forall j \in A^i_k,~i \neq \widehat{\imath},~\widehat{\mu}^k_{s,j}=(1-s) \mu^k_j,~\forall j \in A^{\widehat{\imath}}_k,~\widehat{\mu}^k_{s,N_k +1}=s\left( \sum_{j \in A^{\widehat{\imath}}_k} \mu^k_j \right).
\end{align}
Set~$\widehat{u}_{k,s}=B(\widehat{u}'_{k,s},c_k)$. By construction, there holds
\begin{align*}
\widehat{r}_j(u_{k+1})-\widehat{r}_j(u_{k}) \leq F(K\widehat{u}_{k,s})-F(K\widehat{u}_{k})+ \beta\left(\sum^{N_k +1}_{j=1} \widehat{\mu}^k_{s,j}- \sum^{N_k}_{j=1} \mu^k_j \right).
\end{align*}
The following properties of~$\widehat{u}_{k,s}$ follow directly.
\begin{lemma} \label{lem:propsofhatu}
Let~$u_k$ and~$p_k$ be generated by Algorithm~\ref{alg:accgcg}. Moreover let~$\widehat{u}_{k,s}$ be defined as above. Then there holds
\begin{align*}
\langle p_k, u'_k- \widehat{u}'_{k,s}  \rangle= -s \left( \sum_{j \in A^{\widehat{\imath}}_k} \mu^k_j \right) (\cnorm{p_k}-\beta),~  \sum^{N_k +1}_{j=1} \widehat{\mu}^k_{s,j}= \sum^{N_k}_{j=1} \mu^k_j.
\end{align*}
\end{lemma}
\begin{proof}
Note that
\begin{align*}
u'_k- \widehat{u}'_{k,s}= s\sum_{j \in A^{\widehat{\imath}}_k} \mu^k_j v^k_j- s\left( \sum_{j \in A^{\widehat{\imath}}_k} \mu^k_j \right)\widehat{v}^k
\end{align*}
and thus
\begin{align*}
\langle p_k, u'_k- \widehat{u}'_{k,s}  \rangle= s \left( \sum_{j \in A^{\widehat{\imath}}_k} \mu^k_j \langle p_k, v^k_j \rangle- \left( \sum_{j \in A^{\widehat{\imath}}_k} \mu^k_j \right) \langle p_k, \widehat{v}^k \rangle \right)=-s \left( \sum_{j \in A^{\widehat{\imath}}_k} \mu^k_j \right) (\cnorm{p_k}-\beta)
\end{align*}
using that~$\langle p_k, v^k_j \rangle=\beta$, see, and~$\langle p_k,\widehat{v}^k \rangle=\cnorm{p_k}$. The statement on~$\sum^{N_k +1}_{j=1} \widehat{\mu}^k_{s,j}$ is imminent.
\end{proof}
As a final step we now use~$\widehat{u}_{k,s}$ to prove a refined descent estimate for Algorithm~\ref{alg:accgcg}.
\begin{lemma} \label{lem:fundamentallemma}
For all~$k\in\N$ large enough there holds
\begin{align} \label{eq:fundemantallemma}
\widehat{r}_j(u_{k+1})-\widehat{r}_j(u_{k}) \leq \min_{s \in[0,1]} \left\lbrack \left( s^2c_1-s \left( \min_{i=1,\dots,N} \bar{\mu}_i/2M_0 \right) \right) \widehat{r}_j(u_{k}) \right\rbrack
\end{align}
for some~$c_1>0$ independent of~$k$ and~$s$.
\end{lemma}
\begin{proof}
Let~$s\in [0,1]$ be arbitrary but fixed. We estimate
\begin{align*}
\widehat{r}_j(u_{k+1})-\widehat{r}_j(u_{k}) \leq F(K\widehat{u}_{k,s})-F(K\widehat{u}_{k})+ \beta\left(\sum^{N_k +1}_{j=1} \widehat{\mu}^k_{s,j}- \sum^{N_k}_{j=1} \mu^k_j \right)= F(K\widehat{u}_{k,s})-F(K\widehat{u}_{k})
\end{align*}
where the last equality holds due to Lemma~\ref{lem:propsofhatu}.
As in the proof of Lemma~\ref{lem:descent} we now find
\begin{align}
F(K\widehat{u}_{k,s})-F(K\widehat{u}_{k}) \leq \langle p_k,u'_k- \widehat{u}'_{k,s} \rangle+ \frac{L}{2} \ynorm{K(u_k-\widehat{u}_{k,s})}^2
\end{align}
where~$L>0$ denotes the Lipschitz constant of~$\nabla F$. Summarizing the previous observations and again utilizing Lemma~\ref{lem:propsofhatu} we thus get
\begin{align*}
\widehat{r}_j(u_{k+1})-\widehat{r}_j(u_{k}) \leq -s \left( \sum_{j \in A^{\widehat{\imath}}_k} \mu^k_j \right) (\cnorm{p_k}-\beta)+ \frac{L}{2} \ynorm{K(u_k-\widehat{u}_{k,s})}^2.
\end{align*}
Next we use~\eqref{upperforres} as well as~$\sum_{j \in A^i_k} \mu^k_j \rightarrow \bar{\mu}_i$,~$i=1,\dots,N$, see Lemma~\ref{lem:convlocalized}, to establish the upper bound
\begin{align*}
-s \left( \sum_{j \in A^{\widehat{\imath}}_k} \mu^k_j \right) (\cnorm{p_k}-\beta) \leq -s \left(\min_{i=1,\dots,N} \bar{\mu}_i/2M_0 \right) \widehat{r}_j(u_k).
\end{align*}
Finally it remains to estimate the difference of the observations associated to~$\widehat{u}_{k,s}$ and~$u_k$, respectively. For this purpose we note
\begin{align*}
\ynorm{K(u_k-\widehat{u}_{k,s})} &\leq s \sum_{j \in A^{\widehat{\imath}}_k} \mu^k_j \ynorm{\mathcal{K}(v^k_j-\widehat{v}^k,0)} \\ & \leq s \sum_{j \in A^{\widehat{\imath}}_k} \mu^k_j \left( \ynorm{\mathcal{K}(v^k_j-\bar{u}_{\widehat{\imath}},0)} +\ynorm{\mathcal{K}(\bar{u}_{\widehat{\imath}}-\widehat{v}^k,0)} \right) \\ & \leq
s \sum_{j \in A^{\widehat{\imath}}_k} \mu^k_j \left( |t^k_j-\bar{t}_i|+|\mathbf{v}^k_j-\bar{\mathbf{v}}_i|_{\R^d} +|\hat{t}_k-\bar{t}_{\widehat{\imath}}|+|p_k(\hat{t}_k)/\cnorm{p_k}-\bar{\mathbf{v}}_{\widehat{\imath}}|_{\R^d} \right) \\ & \leq s c\left( 1+ \left( \sum_{j \in A^{\widehat{\imath}}_k} \mu^k_j \right) \right) \sqrt{\widehat{r}_j(u_k)}
\end{align*}
where Lemma~\ref{lem:lipofcalK} is used in the third inequality and Lemma~\ref{lem:quantit} as well as Lemma~\ref{prop:localmaximizers} in the final one. Again pointing out that~$\sum_{j \in A^{\widehat{\imath}}_k} \mu^k_j$ is uniformly bounded independently of~$\widehat{\imath}$ and~$k \in \N$, see Lemma~\ref{lem:convlocalized}, we finally arrive at
\begin{align*}
\ynorm{K(u_k-\widehat{u}_{k,s})}^2 \leq s^2\, c\, \widehat{r}_j(u_k)
\end{align*}
and thus
\begin{align*}
\widehat{r}_j(u_{k+1})-\widehat{r}_j(u_{k}) \leq  \left( s^2c-s \left( \min_{i=1,\dots,N} \bar{\mu}_i/2M_0 \right) \right) \widehat{r}_j(u_{k}).
\end{align*}
Minimizing both sides w.r.t~$s\in[0,1]$ yields the desired result.
\end{proof}
Using this improved descent estimate we prove the linear convergence of the auxiliary residual~$\widehat{r}_j(u_k)$.
\begin{theorem} \label{thm:linconvergence}
Let Assumptions~\ref{ass:problem}-\ref{ass:nondegeneratesource2} hold. Then there is~$\zeta \in (0,1)$ such that
\begin{align*}
r_j (u_k) \leq \widehat{r}_j (u_k) \leq c \zeta^{k}
\end{align*}
for all~$k\in\N$ large enough.
\end{theorem}
\begin{proof}
According to Lemma~\ref{lem:fundamentallemma} there is~$K\in\N$ such that
\begin{align*}
\widehat{r}_j(u_{k+1}) \leq \min_{s \in[0,1]} \left\lbrack \left(1+ s^2c_1-s c_2 \right) \widehat{r}_j(u_{k}) \right\rbrack \quad \forall k\geq K
\end{align*}
where we set
\begin{align*}
c_2 \coloneqq \left( \min_{i=1,\dots,N} \bar{\mu}_i/2M_0 \right)
\end{align*}
for abbreviation. Explicitly calculating the minimum reveals
\begin{align*}
\min_{s \in[0,1]}  \left(1+ s^2c_1-s c_2 \right) \leq  \zeta \coloneqq 1- \frac{c_2}{2} \min \left\{1, \frac{c_2}{2c_1}\right\}
\end{align*}
and thus
\begin{align*}
r_j(u_k)\leq\widehat{r}_j(u_{k+1}) \leq \zeta^{k-K} \widehat{r}_j(u_K)
\end{align*}
for all~$k\geq K$.
\end{proof}
\subsubsection*{ Linear convergence of the iterates}
In this last subsection we aim to quantify the strict convergence of~$u_k$ towards~$\optu$. More in detail we utilize Theorem~\ref{thm:linconvergence} to prove
\begin{align}\label{eq:strictest}
\|u_k-\optu\|_{L^1}+ |\mnorm{u'_k}-\mnorm{\optu'}| \leq c \zeta^k_2
\end{align}
for some~$\zeta_2 \in (0,1)$ and all~$k \in \N $ large enough.
For this purpose we rely on the following auxiliary estimates.
\begin{lemma} \label{lem:estofnorms}
For all~$k\in\N$ large enough there holds
\begin{align*}
|\mnorm{u'_k}-\mnorm{\optu'}|\leq c \sqrt{\widehat{r}_j(u_k)}+  \sum^N_{i=1} \big| \sum_{j \in A^i_k}\mu^k_j-\bar{\mu}_i \big| \Big \rbrack
\end{align*}
\end{lemma}
\begin{proof}
Recall the definition of the restricted measures~$U'_{k,i}$ from~\eqref{def:restricmeasures}. Then there holds
\begin{align*}
|\mnorm{u'_k}-\mnorm{\optu'}| \leq \sum^N_{i=1} \Big \lbrack \big|\mnorm{U'_{k,i}}-\sum_{j \in A^i_k}\mu^k_j\big|+ \big| \sum_{j \in A^i_k}\mu^k_j-\bar{\mu}_i \big| \Big \rbrack
\end{align*}
Now, fix an arbitrary~$i \in \{1, \dots,N\}$. Given two indices~$j_1,j_2 \in A^i_k$ we note that
\begin{align*}
\mnorm{ \mu^k_{j_1}u^k_{j_1}+\mu^k_{j_2}u^k_{j_2}}= \mu^k_{j_1}+\mu^k_{j_2}
\end{align*}
if~$t^k_{j_1}\neq t^k_{j_2}$ and
\begin{align*}
\mnorm{ \mu^k_{j_1}u^k_{j_1}+\mu^k_{j_2}u^k_{j_2}}= \big|\mu^k_{j_1}\mathbf{u}^k_{j_1}+\mu^k_{j_2}\mathbf{u}^k_{j_2}\big|_{\R^d}
\end{align*}
if~$t^k_{j_1}= t^k_{j_2}$.
Similarly we conclude the existence of a partition of~$A^i_k$ into pairwise disjoint, nonempty sets~$I^h_k$,~$h=1,\dots,n_k$, with
\begin{align*}
\big|\mnorm{U'_{k,i}}-\sum_{j \in A^i_k}\mu^k_j\big|&= \left| \sum^{n_k}_{h=1} \Big \lbrack \big | \sum_{j \in I^h_k} \mu^k_j \mathbf{v}^k_j \big|_{\R^d}- \sum_{j \in I^h_k} \mu^k_j |\bar{\mathbf{v}}_i|_{\R^d} \Big \rbrack\right| \\ & \leq \sum^{n_k}_{h=1} \Big|\big | \sum_{j \in I^h_k} \mu^k_j \mathbf{v}^k_j \big|_{\R^d}- \sum_{j \in I^h_k} \mu^k_j |\bar{\mathbf{v}}_i|_{\R^d} \Big|
\\ & \leq
\sum^{n_k}_{h=1} \Big| \sum_{j \in I^h_k} \mu^k_j( \mathbf{v}^k_j-\bar{\mathbf{v}}_i) \Big|_{\R^d}
\\& \leq
\sum^{n_k}_{h=1} \sum_{j \in I^h_k} \mu^k_j \big | \mathbf{v}^k_j-\bar{\mathbf{v}}_i \big|_{\R^d}
\\& \leq c \sqrt{\widehat{r}_j(u_k)}
\end{align*}
where we use the inverse triangle inequality in the second inequality and
\begin{align*}
\sum^{n_k}_{h=1} \sum_{j \in I^h_k} \mu^k_j \big | \mathbf{v}^k_j-\bar{\mathbf{v}}_i \big|_{\R^d} = \sum_{j \in A^i_k} \mu^k_j \big | \mathbf{v}^k_j-\bar{\mathbf{v}}_i \big|_{\R^d}
\end{align*}
as well as Lemma~\ref{lem:quantit} in the final inequality. Summarizing all previous observations and noting that the index~$i$ was chosen arbitrarily finishes the proof.
\end{proof}
A similar estimate holds for the $L^1$ distance of the iterates to the minimizer~$\optu$.
\begin{lemma} \label{lem:estofl1}
Define constants
\begin{align} \label{defofCk}
\bar{C}=- \frac{1}{T}\int_0^T~ \int_0^s \de \bar{u}'~\de s +a_{\bar{u}},~ C^k=- \frac{1}{T}\int_0^T~ \int_0^s \de u'_k~\de s +a_{u_k}.
\end{align}
For all~$k\in\N$ large enough there holds
\begin{align*}
\|u_k-\optu\|_{L^1} \leq c \sqrt{\widehat{r}_j(u_k)}+ T |C^k-\bar{C}|+\sum^N_{i=1} \Big| \sum_{j \in A^i_k} \mu^k_i -\bar{\mu}_i \Big|
\end{align*}
\end{lemma}
\begin{proof}
According to the definition of the $B$ operator,~\eqref{def:Boperator}, we have
\begin{align*}
\bar{u}=\bar{C} + \sum^N_{i=1} \bar{\mu}_i \bar{\mathbf{v}}_i \chi_{\bar{t}_i},~u_k= C^k + \sum^N_{i=1} \sum_{j \in A^i_k} \mu^k_j \mathbf{v}^k_j \chi_{t^k_j}
\end{align*}
and thus
\begin{align*}
\|u_k-\optu\|_{L^1} &\leq T|C^k-\bar{C}| + \sum^N_{i=1} \left\|\sum_{j \in A^i_k} \mu^k_j \mathbf{v}^k_j \chi_{t^k_j}-\bar{\mu}_i \bar{\mathbf{v}}_i\chi_{\bar{t}_i}\right\|_{L^1}.
\end{align*}
Now fix an arbitrary index~$i \in \{1, \dots,N\}$. We estimate
\begin{align*}
\left\|\sum_{j \in A^i_k} \mu^k_j \mathbf{v}^k_j \chi_{t^k_j}-\bar{\mu}_i \bar{\mathbf{v}}_i\chi_{\bar{t}_i}\right\|_{L^1} &\leq T \Big| \sum_{j \in A^i_k} \mu^k_i -\bar{\mu}_i \Big| + \left\|\sum_{j \in A^i_k} \mu^k_j (\mathbf{v}^k_j \chi_{t^k_j}- \bar{\mathbf{v}}_i\chi_{\bar{t}_i})\right\|_{L^1}
\\ & \leq
T \Big| \sum_{j \in A^i_k} \mu^k_i -\bar{\mu}_i \Big| + \sum_{j \in A^i_k} \mu^k_j\| \mathbf{v}^k_j \chi_{t^k_j}- \bar{\mathbf{v}}_i\chi_{\bar{t}_i}\|_{L^1}
\end{align*}
using that~$\left\|\bar{\mathbf{v}}_i \chi_{\bar{t}_i}\right\|_{L^1(I;\R^d)}\leq T$. Moreover, from~\eqref{eq:estforindic} and \eqref{eq:l2likeest}, we conclude
\begin{align*}
\sum_{j \in A^i_k} \mu^k_j\| \mathbf{v}^k_j \chi_{t^k_j}- \bar{\mathbf{v}}_i\chi_{\bar{t}_i}\|_{L^1} \leq \sum_{j \in A^i_k} \mu^k_j \left( |t^k_j-\bar{t}_i|+T|\mathbf{v}^k_j-\bar{\mathbf{v}}_i|_{\R^d} \right) \leq c \sqrt{\widehat{r}_j(u_k)}.
\end{align*}
Summarizing all previous observations yields the desired estimate.
\end{proof}
Thus to prove~\eqref{eq:strictest} it suffices to quantify the error~$C_k-\bar{C}$ as well as the difference between~$\sum_{j \in A^i_k} \mu^k_j$ and~$\bar{\mu}_i$. This is done in the following proposition.
\begin{prop} \label{prop:convofjumpheights}
For all~$k\in \N$ large enough there holds
\begin{align*}
|\bar{C}-C^k|+\Big| \sum_{j \in A^i_k} \mu^k_i -\bar{\mu}_i \Big| \leq c \sqrt{\widehat{r}_j(u_k)}.
\end{align*}
\end{prop}
\begin{proof}
Define~$\tilde{u}_k=C_k+ \sum^N_{i=1} \left( \sum_{j \in A^i_k} \mu^k_j \right) \bar{\mathbf{v}}_i \chi_{\bar{t}_i}$ as well as the vector of lumped coefficients~$\tilde{\mu}^k \in \R^N$,~$\tilde{\mu}^k_i=\sum_{j \in A^i_k} \mu^k_j$. Recall the definition of the injective operator~$\widehat{\mathbf{K}}$ from the proof of Corollary~\ref{coroll:uniqueness}. Then~$\widehat{\mathbf{K}}(\tilde{\mu}^k-\bar{\mu}, C^k -\bar{C})=K(u_k-\bar{u})$ and thus
\begin{align*}
|\bar{C}-C^k|+\Big| \sum_{j \in A^i_k} \mu^k_i -\bar{\mu}_i \Big| \leq c \ynorm{K(\tilde{u}_k-\bar{u})}.
\end{align*}
 Applying Proposition~\ref{prop:estforstates} yields
\begin{align*}
\|K(\tilde{u}_k-\bar{u})\|_Y \leq \|K(u_k-\bar{u})\|_Y+\|K(\tilde{u}_k-u_k)\|_Y \leq \sqrt{r_j(u_k)/\gamma_0}+\|K(\tilde{u}_k-u_k)\|_Y.
\end{align*}
Finally we estimate
\begin{align*}
\|K(\tilde{u}_k-u_k)\|_Y &\leq \sum^N_{i=1} \sum_{j\in A^i_k} \mu^k_j \ynorm{\mathcal{K}(v^k_j-\bar{u}_i,0)} \\&\leq c \sum^N_{i=1} \sum_{j\in A^i_k} \mu^k_j \left( |t^k_j-\bar{t}_i|+|\mathbf{v}^k_j-\bar{\mathbf{v}}_i|_{\R^d} \right) \\&\leq c \sqrt{\widehat{r}_j(u_k)}
\end{align*}
using Lemma~\ref{lem:lipofcalK} in the second inequality and Lemma~\ref{lem:quantit} in the final one.
\end{proof}
Combining the previous results we are in the position to prove linear convergence of~$u_k$ with respect to the strict topology on~$\BV$.
\begin{theorem} \label{thm:confofiterates}
Let Assumptions~\ref{ass:assumponF},~\ref{ass:nondegeneratesource1}, \ref{ass:nondegeneratesource2} hold. Then we have
\begin{align*}
\|u_k-\optu\|_{L^1} + |\mnorm{u'_k}-\mnorm{\optu'}| \leq c \zeta_2^k.
\end{align*}
for some~$\zeta_2 \in (0,1)$ and all~$k\in\N$ large enough.
\end{theorem}
\begin{proof}
The statement directly follows from Lemma~\ref{lem:estofnorms} and Lemma~\ref{lem:estofl1} taking Proposition~\ref{prop:convofjumpheights} into account.
\end{proof}
\section{Numerical examples} \label{sec:num}
The last section is devoted to the numerical illustration of our theoretical results. For this purpose two examples are discussed. First we address the inverse problem of identifying a piecewise constant signal from finitely many data samples. The forward operator~$K$ is modelled by convolution with a Gaussian kernel. Second we consider an optimal control problem for the linear wave equation. Here the control enters as the time-dependent signals of two spatially fixed actuators. In this case, the fidelity term is given by the $L^2$-misfit between the solution to the wave equation and a desired state~$y_d$ over the whole space-time cylinder.
\subsection{Deconvolution from finitely many measurements}
As a first example consider
\begin{align} \label{def:convolprob}
\min_{u \in \operatorname{BV}(I)} j(u)  \coloneqq \left \lbrack \frac{1}{2}\sum^{9}_{i=1} (k(\rho_i)* u-y^i_d)^2+\beta \mnorm{u'} \right \rbrack
\end{align}
where~$I=(0,1)$,~$y_d \in \R^9$ is a given finite dimensional data vector and
\begin{align*}
k(\rho_i)* u= \frac{1}{\sqrt{2\pi}\sigma}\int^1_0 u(t)\, e^{-\frac{(t-\rho_i)}{2\sigma^2}^2}~\mathrm{d}t,~\rho_i= i \cdot 0.1, \quad i=1,\dots,9.
\end{align*}
The deconvolution problem~\eqref{def:convolprob} can be embedded in the general setting \eqref{def:bvprob} by choosing
\begin{align*}
Y=\R^9,~  F(\cdot)=\frac{1}{2}\sum^9_{i=1} ((\cdot)_i-y_i)^2,~(Ku)_i=k(\rho_i)* u.
\end{align*}
In this case the~$\mathcal{K}=K \circ B$ operator is given by
\begin{align*}
\mathcal{K}(q,c)_i=(k(\rho_i),B(q,c))_{L^2} =-\langle \psi_i,q \rangle +\psi_i(1)(c+\langle t, q \rangle), \quad i=1,\dots,9,
\end{align*}
for~$\psi_i$ defined as
\begin{align*}
\psi_i(t)=\frac{1}{2}\left( \operatorname{erf} \left(\frac{t-\rho_i}{\sqrt{2}\sigma} \right)+\operatorname{erf} \left(\frac{\rho_i}{\sqrt{2}\sigma} \right) \right)
\end{align*}
where~$\operatorname{erf}(\cdot)$ denotes the error function.
Moreover we readily verify~$\mathcal{K}^* \colon \R^m \to \Cc^2(I)$ and
\begin{align*}
p_k(\cdot)=\int^{\cdot}_0 K^* \nabla F(Ku_k)= \sum^m_{i=1}  \psi_i(\cdot){(k(\rho_i)* u_k-y_i)}.
\end{align*}
In order to determine a global extremum of~$p_k$, see step 3., we find solutions of~$p'_k(t)=0$ using a Newton method starting at equally spaced points~$t^i_0= i \cdot 0.1$,~$i=1,\dots,9$. Then~$\hat{t}_k$ is chosen from the set of computed solutions by comparing the corresponding function values. The solution of the finite dimensional subproblems relies on a semismooth Newton method for the "normal map" reformulation of its first order sufficient optimality conditions, see e.g.~\cite{milzarek}. In each iteration the method is warmstarted using the current magnitudes~$\mu^k_i$ and the mean value~$c^k$ to construct a good starting point. Moreover we further enhance its practical performance by incorporating a heuristic globalization strategy based on damped Newton steps. Finally Algorithm~\ref{alg:accgcg} is stopped if the upper bound
\begin{align*}
\Phi_k \coloneqq M_k(\cnorm{p_k}-\beta),~M_k= F(Ku_k)+\beta \sum^{\# \mathcal{A}_k}_{i=1} \mu^k_i
\end{align*}
on the residual~$r_j(u_k)$, see Theorem~\ref{thm:slowconvalg}, is smaller than~$10^{-13}$.
\subsubsection{Structural assumptions on~$\bar{p}$}
We solved~\eqref{def:convolprob} for~$\beta \approx 10^{-5}$ and observations~$y_d= Ku^\dagger+ \zeta$ where~$u^\dagger \in \operatorname{BV}(I)$ and~$\zeta \in \R^9$ is a random perturbation. The ground truth~$u^\dagger$ and the computed minimizer~$\optu$ are depicted in Figure~\ref{fig:sol}. Before addressing the performance of Algorithm~\ref{alg:accgcg} we numerically verify Assumptions~\ref{ass:nondegeneratesource1} and~\ref{ass:nondegeneratesource2}. For this purpose we plot the dual variable~$\bar{p}$ as well as its second derivative~$\bar{p}{''}$ in Figures~\ref{fig:adjointval}~and~\ref{fig:hessian}. The functional values corresponding to the jumps of~$\optu$ are marked by red crosses.
\begin{figure}[htb]
\begin{subfigure}[t]{.31\linewidth}
\centering
\includegraphics[scale=0.38]{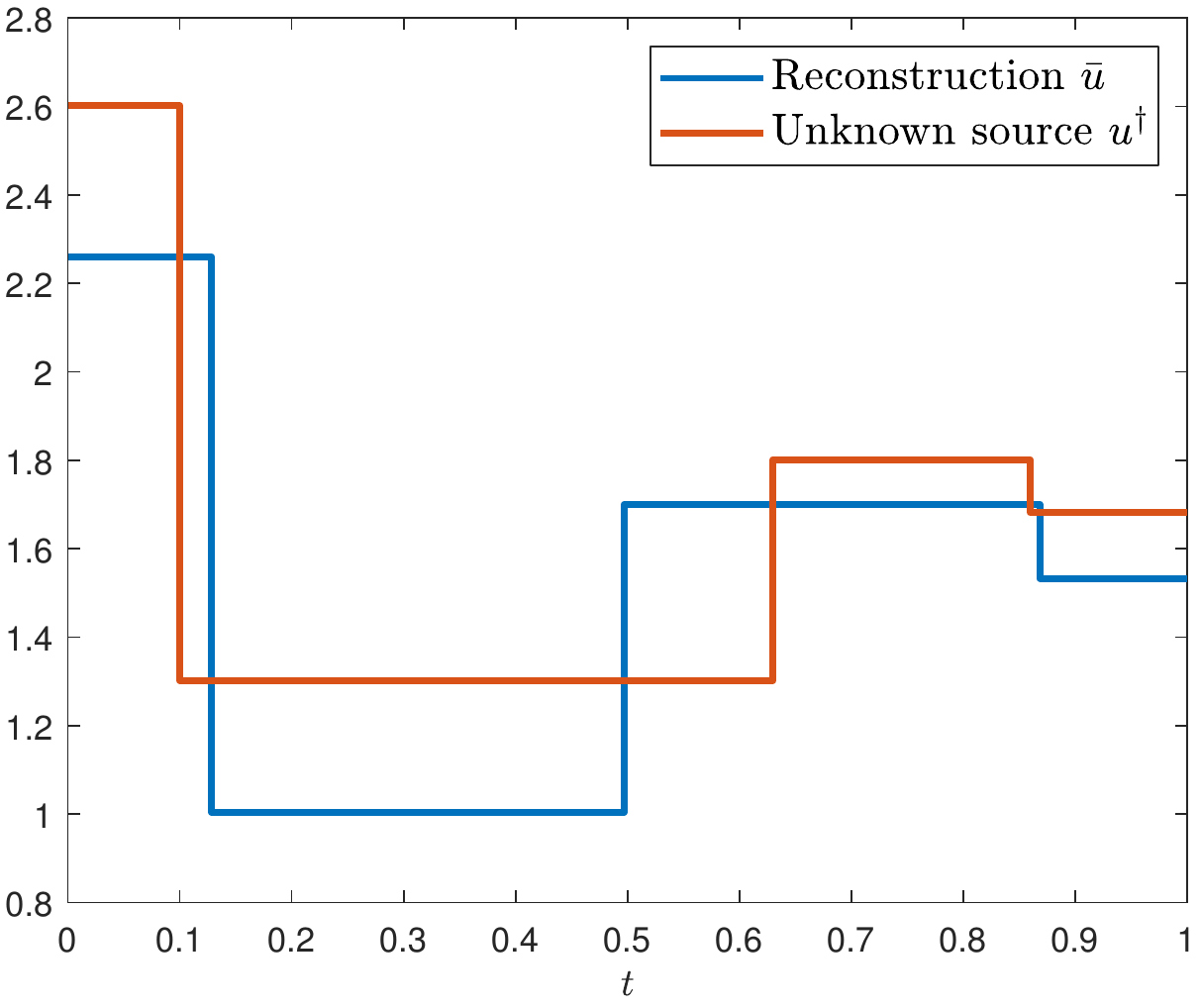}
\caption{Ground truth~$u^\dagger$ and~$\bar{u}$.}
\label{fig:sol}
\end{subfigure}
\quad
\begin{subfigure}[t]{.31\linewidth}
\centering
\includegraphics[scale=0.38]{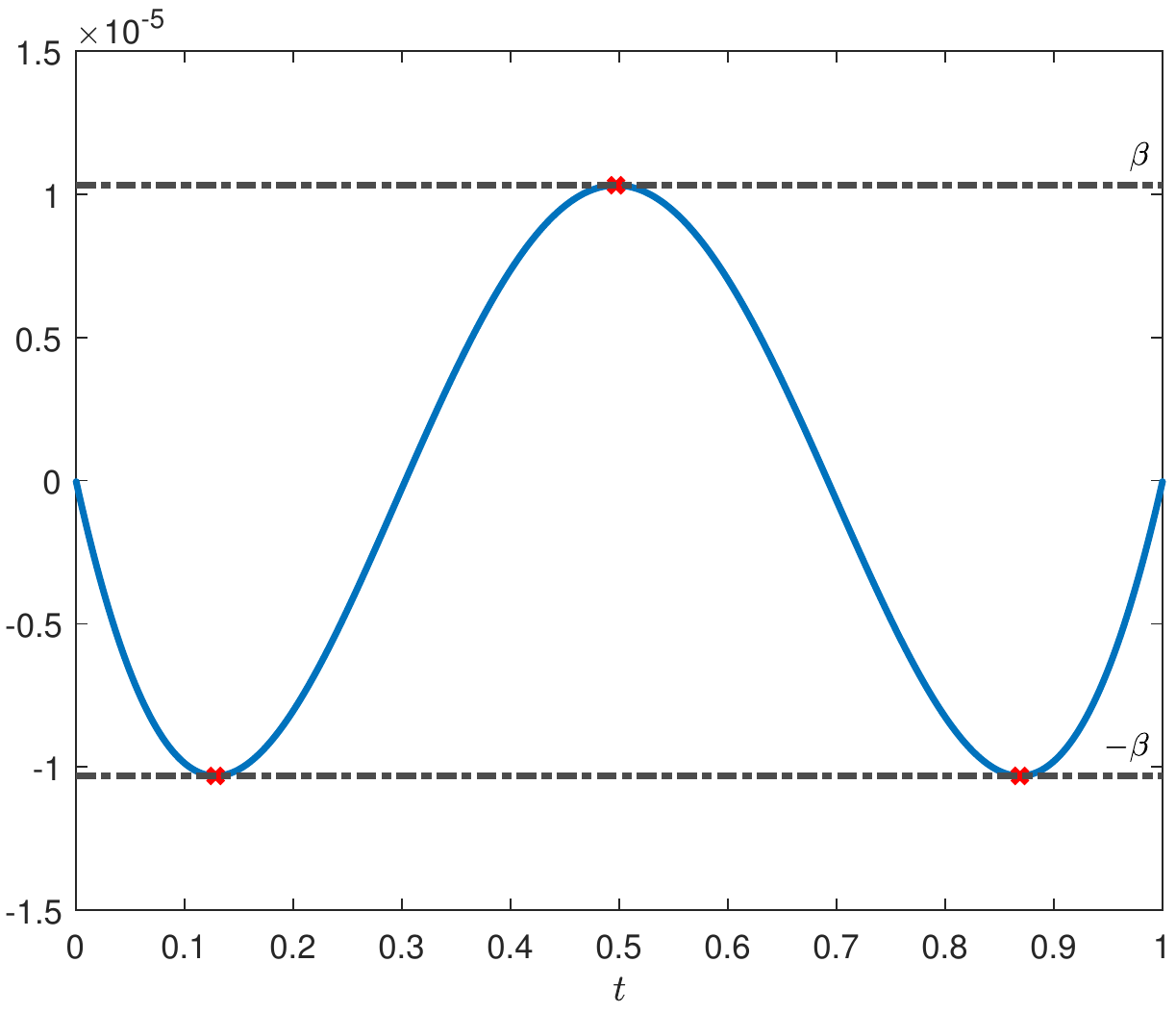}
\caption{Dual variable~$\bar{p}$.}
\label{fig:adjointval}
\end{subfigure}
\quad
\begin{subfigure}[t]{.31\linewidth}
\centering
\includegraphics[scale=0.38]{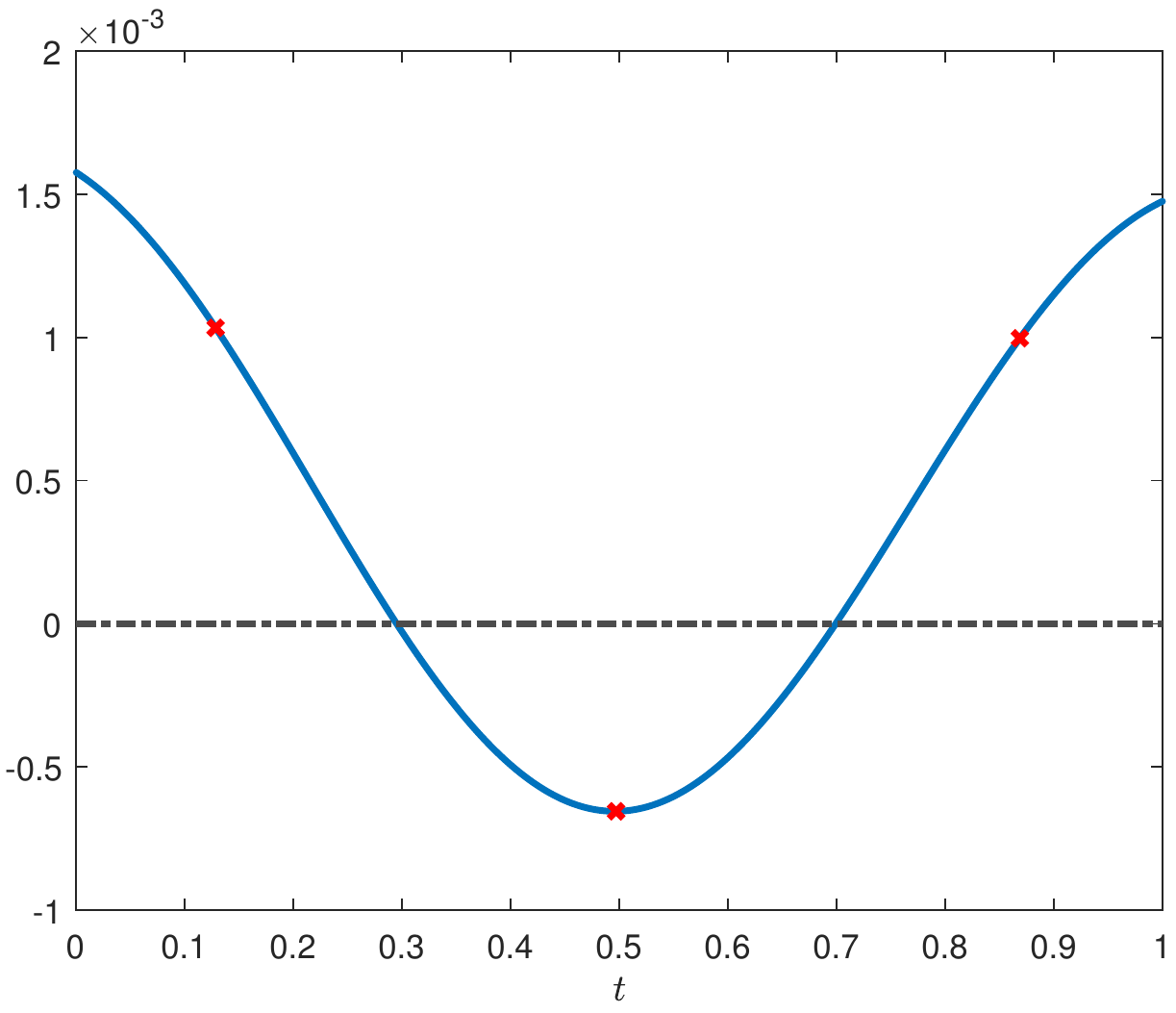}
\caption{Second derivative~$\bar{p}''$.}
\label{fig:hessian}
\end{subfigure}
\caption{Ground truth, reconstruction and dual variable.}
\label{fig:adjoint}
\end{figure}

First we point out that~$\cnorm{\bar{p}}=\beta$ and~$\bar{p}$ achieves its global maximum/minimum in three distinct points~$\{\bar{t}_i\}^{3}_{i=1}$ which coincide with the jumps of~$\bar{u}$. In particular, the optimal magnitudes satisfy~$\bar{\mu}_i>0$. Moreover the operator~$\widehat{\mathbf{K}}$ from the proof of Corollary~\ref{coroll:uniqueness} has full rank which is equivalent to the linear independence of~\eqref{eq:linindass}. Second, there holds~$\bar{p}''(\bar{t}_i)\neq 0$. Hence, see Remark~\ref{rem:quadgrowth}, the quadratic growth condition of Assumption~\ref{ass:nondegeneratesource2} holds.
\subsubsection{Practical performance of Algorithm~\ref{alg:accgcg}} \label{subsubsec:practical}
In order to assess the performance of Algorithm~\ref{alg:accgcg} we plot the residuals~$r_j(u_k)$ alongside the sublinear convergence rate from Theorem~\ref{thm:slowconvalg} as well as a linear rate with~$\zeta=0.33$ in Figure~\ref{fig:res}.
\begin{figure}[htb]
\begin{subfigure}[t]{.31\linewidth}
\centering
\scalebox{.38}{
%
%
\definecolor{mycolor1}{rgb}{0.00000,0.44700,0.74100}%
\definecolor{mycolor2}{rgb}{0.85000,0.32500,0.09800}%
\begin{tikzpicture}

\begin{axis}[%
width=4.521in,
height=3.566in,
at={(0.758in,0.481in)},
scale only axis,
xmin=1,
xmax=11,
xlabel style={font=\color{white!15!black}},
xlabel={$k$},
ymode=log,
ymin=1e-18,
ymax=0.0001,
yminorticks=true,
axis background/.style={fill=white},
legend style={at={(0.03,0.03)}, anchor=south west, legend cell align=left, align=left, draw=white!15!black}
]
\addplot [color=mycolor1, mark=diamond, mark options={solid, mycolor1}, line width=2pt]
  table[row sep=crcr]{%
1	8.56136447512647e-05\\
2	4.04139462522138e-05\\
3	7.28031473574152e-06\\
4	3.91392054451726e-06\\
5	3.0036322289876e-07\\
6	2.78867312658149e-07\\
7	1.06631927949789e-07\\
8	4.45068773502476e-11\\
9	2.90800833933901e-11\\
10	2.09272703333151e-14\\
11	4.33680868994202e-18\\
};
\addlegendentry{$r_j(u_k)$}

\addplot [color=mycolor2, dashdotted,, line width=2pt]
  table[row sep=crcr]{%
1	8.56136447512644e-05\\
2	4.28068223756322e-05\\
3	2.85378815837548e-05\\
4	2.14034111878161e-05\\
5	1.71227289502529e-05\\
6	1.42689407918774e-05\\
7	1.22305206787521e-05\\
8	1.0701705593908e-05\\
9	9.51262719458492e-06\\
10	8.56136447512644e-06\\
11	7.7830586137513e-06\\
};
\addlegendentry{$\mathcal{O}(1/k)$}

\addplot [color=black, dashed, line width=2pt]
  table[row sep=crcr]{%
1	8.56136447512644e-05\\
11	1.23438351074894e-09\\
};
\addlegendentry{$\mathcal{O}(0.33^k)$}

\end{axis}
\end{tikzpicture}%
}
\caption{Residual error over $k$.}
\label{fig:res}
\end{subfigure}
\quad
\begin{subfigure}[t]{.31\linewidth}
\centering
\scalebox{.38}{
\centering
%
%
\definecolor{mycolor1}{rgb}{0.00000,0.44700,0.74100}%
\definecolor{mycolor2}{rgb}{0.85000,0.32500,0.09800}%
\begin{tikzpicture}

\begin{axis}[%
width=4.521in,
height=3.566in,
at={(0.758in,0.481in)},
scale only axis,
xmin=1,
xmax=11,
xlabel style={font=\color{white!15!black}},
xlabel={$k$},
ymode=log,
ymin=1e-10,
ymax=100,
yminorticks=true,
axis background/.style={fill=white},
legend style={legend cell align=left, align=left, draw=white!15!black}
]
\addplot [color=mycolor1, mark=diamond, mark options={solid, mycolor1}, line width=2pt]
  table[row sep=crcr]{%
1	0.368870083210413\\
2	0.339139563842634\\
3	0.318087988861314\\
4	0.255008855845176\\
5	0.128137527811556\\
6	0.073175325921935\\
7	0.0422901583983521\\
8	0.000774667206797643\\
9	0.000660807071958562\\
10	1.87784362175071e-05\\
11	1.17571498409281e-07\\
};
\addlegendentry{$\|u_k-\bar{u}\|_{L^1(I)}$}

\addplot [color=mycolor2, mark=square, mark options={solid, mycolor2}, line width=2pt]
  table[row sep=crcr]{%
1	1.90897063276989\\
2	1.96698048147293\\
3	1.20033776962929\\
4	1.11079637345953\\
5	0.643284943757642\\
6	0.278217718118148\\
7	0.131244744698602\\
8	0.000145394822138378\\
9	5.89831465718227e-05\\
10	1.27937253457944e-06\\
11	7.46854089683779e-09\\
};
\addlegendentry{$|\|u_k '\|_{\mathcal{M}}-\|\bar{u} '\|_{\mathcal{M}}|$}

\addplot [color=black, dashed, line width=2pt]
  table[row sep=crcr]{%
1	2\\
11	0.0230763856080623\\
};
\addlegendentry{$\mathcal{O}(0.64^k)$}

\end{axis}
\end{tikzpicture}%
}
\caption{Norm/$L^1$-error over~$k$.}
\label{fig:stricttop}
\end{subfigure}
\quad
\begin{subfigure}[t]{.31\linewidth}
\centering
\scalebox{.38}{

\centering
%
%
\definecolor{mycolor1}{rgb}{0.00000,0.44700,0.74100}%
\begin{tikzpicture}

\begin{axis}[%
width=4.521in,
height=3.566in,
at={(0.758in,0.481in)},
scale only axis,
xmin=1,
xmax=12,
xlabel style={font=\color{white!15!black}},
xlabel={$k$},
ymin=0,
ymax=4,
ytick={0, 1, 2, 3, 4},
axis background/.style={fill=white},
legend style={legend cell align=left, align=left, draw=white!15!black}
]
\addplot [color=mycolor1, line width=2pt]
  table[row sep=crcr]{%
1	1\\
3	1\\
4	2\\
6	2\\
7	3\\
12	3\\
};
\addlegendentry{$\# \mathcal{A}_k$}

\end{axis}
\end{tikzpicture}%
}
\caption{Active set size over~$k$.}
\label{fig:support}
\end{subfigure}
\caption{Convergence behaviour of relevant quantities.}
\end{figure}
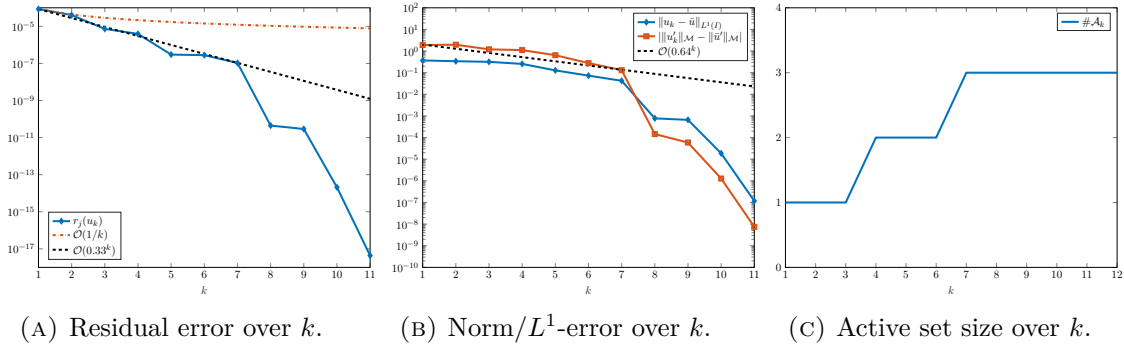
Next to it, in Figure~\ref{fig:stricttop}, we report on the convergence of the iterates~$u_k$ in~$L^1(I)$ and the norms~$\mnorm{u'_k}$. As predicted by Theorems~\ref{thm:linconvergence} and~\ref{thm:confofiterates} all considered quantities converge at least linearly. Moreover we plot the evolution of the size of the active set in Figure~\ref{fig:support}. Note that $\# \mathcal{A}_k$ is not strictly increasing. This is testament to the efficiency of the pruning step 7. of Algorithm~\ref{alg:accgcg} in combination with the full resolution of the subproblem~\eqref{def:subprobjumps} in step 5. Finally we compare Algorithm~\ref{alg:accgcg} to the~\textit{Fast iterative shrinkage-thresholding algorithm} (FISTA) from~\cite{beckfista,chambollefista}. However, in contrast to our proposed method, its practical application to~\eqref{def:convolprob} requires a discretization of the interval~$(0,1)$. For this purpose we consider a uniform partition of~$[0,1]$ into subintervals~$[t_i,t_{i+1}]$,~$i=1, \dots, N_h-1$, where~$t_0=0$ and~$t_i=t_{i-1}+h$, else, with~$h=1/(N_h -1)$. Subsequently we replace~$\operatorname{BV}(I)$ in~\eqref{def:convolprob} by the finite-dimensional subspace
\begin{align*}
\operatorname{BV}_{h}(I)=\left\{\,u \in \operatorname{BV}(I) \;|\;u=B \left(\sum^{N_h-1}_{i=1}\mu^h_i \delta_{t^h_i},c\right),~\mu\in \R^{N_h -1},~c\in \R \,\right\}
\end{align*}
and apply FISTA with constant stepsize as described in~\cite{beckfista}. Additionally we also use this comparison to study the behaviour of Algorithm~\ref{alg:accgcg} under perturbations and apply it to the discretized problem. In this context, we restrict the search for the new candidate jump position~$\hat{t}_k$ in step 3. of Algorithm~\ref{alg:accgcg} to the set of nodes of the partition. More in detail we choose~$\hat{t}_k \in \{t^h_i\}^{N_h-1}_{i=1}$ such that
\begin{align*}
|p_k(\hat{t}_k)|=\max_{i=1,\dots, N_h -1} |p_k(t^h_i)|.
\end{align*}
The other steps of the method remain the same.
\begin{figure}[htb]
\begin{subfigure}[t]{.48\linewidth}
\centering
\includegraphics[scale=0.5]{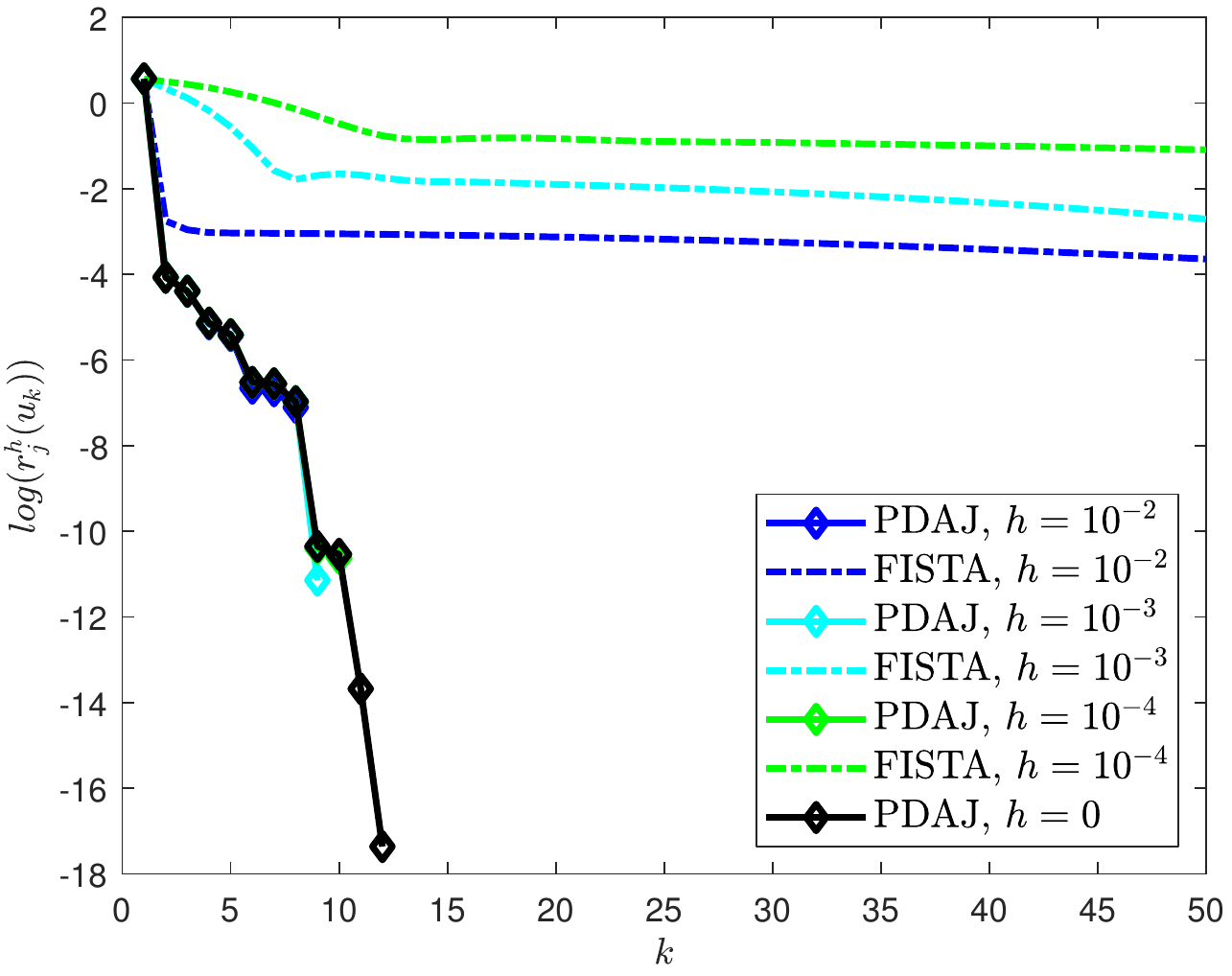}
\caption{Residual vs. iter. number.}
\label{fig:meshinditer}
\end{subfigure}
\quad
\begin{subfigure}[t]{.48\linewidth}
\centering
\includegraphics[scale=0.5]{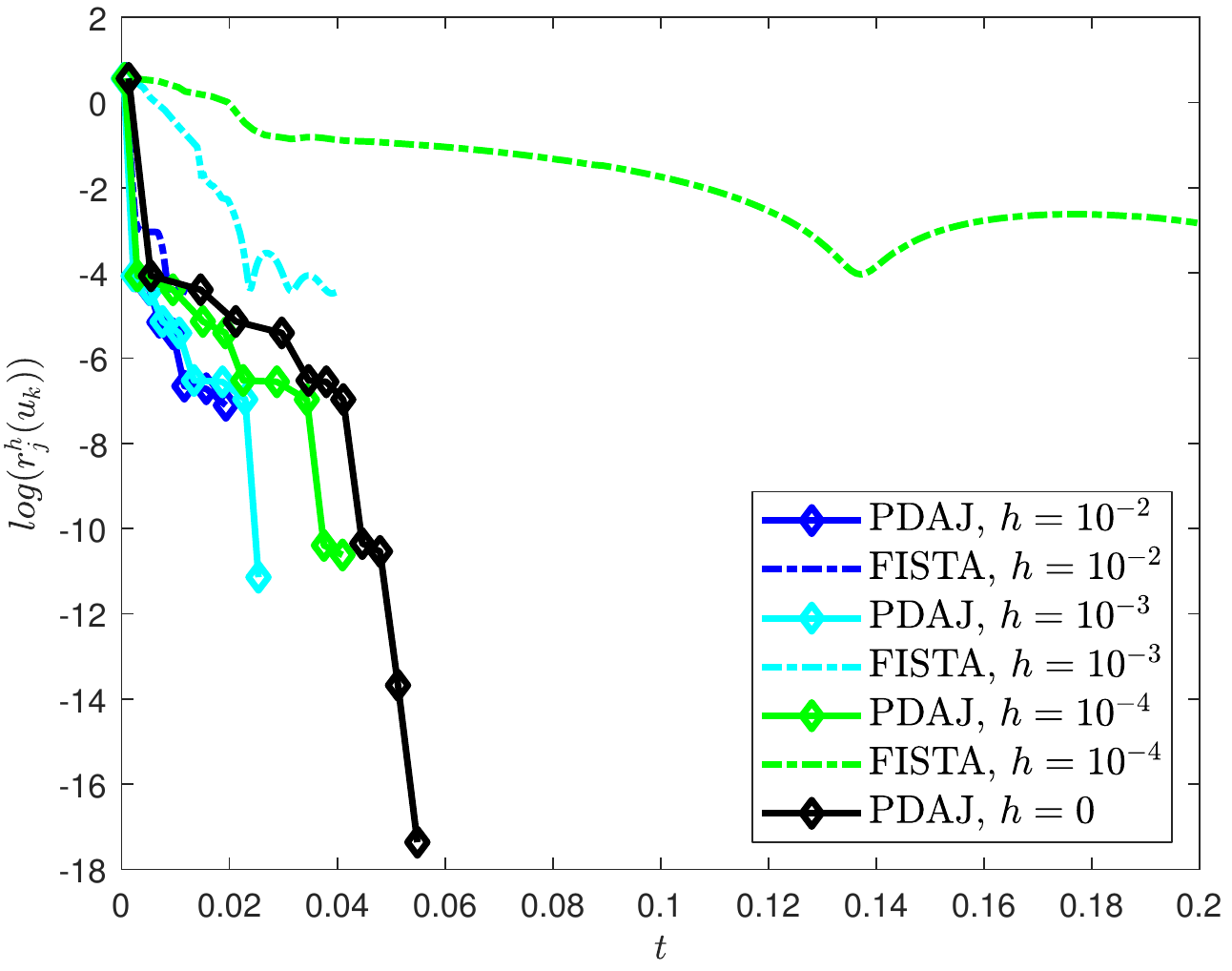}
\caption{Residual vs. comp. time in~seconds.}
\label{fig:meshindtime}
\end{subfigure}
\caption{Comparison of Algorithm~\ref{alg:accgcg} and FISTA on different grids.}
\label{fig:meshind}
\end{figure}
In Figure~~\ref{fig:meshinditer} we plot the behaviour of the residual~$r_j(u_k)=j(u_k)-\min_{u \in \operatorname{BV}_h(I)} j(u)$ for FISTA and our method with different grid widths~$h=10^{-2},10^{-3},10^{-1}$. Additionally we also include Algorithm~\ref{alg:accgcg}~\textit{without} discretization in the plot. This is formally denoted by~$"h=0"$. In both methods, the same starting point~$u_0$ is used. We observe that Algorithm~\ref{alg:accgcg} solves the problem on each refinement level in a few iterations while the convergence of FISTA significantly slows down after the first iterations. Moreover Algorithm~\ref{alg:accgcg} exhibits strong mesh-independence i.e. its convergence is stable w.r.t. to~$h$ and is essentially governed by its behaviour on the continuous problem. In contrast, the convergence behaviour of FISTA degenerates as~$h$ gets smaller. Let us however point out that the per iteration cost of both algorithms is wildly different. In fact, the practical realization of FISTA only requires the computation of one proximal operator per iteration, which can be done analytically, while Algorithm~\ref{alg:accgcg} relies on determining a global extremum of~$p_k$ as well as the full resolution of~\eqref{def:subprobjumps}. To respect the different cost per iteration of both methods we also give a comparison in terms of the computational time in Figure~\ref{fig:meshindtime}. For this purpose we plot the convergence history of Algorithm~\ref{alg:accgcg} (up to optimality) and of FISTA (first 200 iterations) as a function of time. We observe that the more complicated subproblems in Algorithm~\ref{alg:accgcg} do~\textit{not} lead to highly increased computational times. This is, on the one hand, a consequence of the use of an efficient second order optimization scheme for~\eqref{def:subprobjumps} in combination with a warmstart. On the other hand this is also attributed to the observation that the active set size~$\# \mathcal{A}_k$, and thus the dimension of the subproblems~\eqref{def:subprobjumps}, is essentially independent of the underlying discretization. We omit an additional plot showcasing the convergence of~$\# \mathcal{A}_k$ on the different discretization levels since the resulting curves align themselves with the plot in Figure~\ref{fig:support}.
\subsection{Optimal control of the wave equation}
In this section we apply the proposed method for the solution of a PDE-constraint optimization problem of the form
\begin{equation}\label{eq:num_prob}
\min_{u\in BV(I;\mathbb R^2),y \in L^2(I \times \Omega)}J(u)=\frac 12\|y-y_d\|_{L^2(I\times \Omega)}^2+\beta \|u'\|_{\M(I;\mathbb R^2)}
\end{equation}
where the vector-valued control~$u$ is connected to the state variable~$y$ by a linear wave equation of the form
\begin{equation} \label{eq:wave-equation}
    \left\{\begin{aligned}
            \partial_{tt} y - \Delta y &= u_1(t)\delta_{x_1}(x)+u_2(t)\delta_{x_2}(x) && \text{ in } I \times \Omega, \\
            {\partial_\nu y} &= 0 && \text{ on } (0,T)\times \partial \Omega, \\
            y(0) &= 0,\quad {\partial_t} y(0) = 0, && \text{ on }\Omega.
    \end{aligned}\right.
\end{equation}
with $x_1=(0.5,0.5)$, $x_2=(-0.5,-0.5)$, $\beta=10^{-5}$, $\Omega=(-1,1)^2$ and $T=1$. The desired state~$y_d \in L^2(I \times \Omega)$ is given by~$y_d=y^*_d+ \zeta$ where~$y^*_d$ is the unique solution of~\eqref {eq:wave-equation} for the reference source~$u^*=(u^*_1,u^*_2)$ given by 
\[
u^*_1(t)=
\begin{cases}
0.05 &0<t\leq 0.25,\\
0.65 &0.25 < t\leq0.5,\\
0.15 &0.5 < t \leq0.75,\\
0.35 &0.75 < t \leq 1
\end{cases}
\quad
u^*_2(t)=
\begin{cases}
0.775 &0<t\leq 0.25,\\
-0.025 &0.25 < t\leq0.5,\\
0.975 &0.5 < t \leq0.75,\\
0.275 &0.75 < t \leq 1
\end{cases}
\]
and~$\zeta \in L^2(I\times \Omega)$,~$\|\zeta\|_{L^2(I\times \Omega)}/\|y^*_d\|_{L^2(I\times \Omega)}=0.05$, is a noise term.

Using the $L^2(I\times \Omega)$ regularity of $y$ from \cite{Triggiani91,kunisch2016optimal} we can eliminate the PDE-constraint by introducing the linear continuous solution operator $K\in L(L^2(I;\R^2),L^2(I\times \Omega))$ which maps $u$ to $y$.
The adjoint operator $K^\ast\in L(L^2(I \times \Omega),L^2(I;\R^2))$ of $K$ is defined by the mapping $\phi \mapsto (p(\cdot,x_1), p(\cdot,x_2))$ where $p$ is the solution of the corresponding adjoint state equation
\begin{equation} \label{eq:wave-equation-adjoint}
    \left\{\begin{aligned}
            \partial_{tt} p - \Delta p &= \phi && \text{ in } I \times \Omega, \\
            {\partial_\nu p} &= 0 && \text{ on } (0,T)\times \partial \Omega, \\
            p(T) &= 0,\quad {\partial_t} p(T) = 0, && \text{ on }\Omega.
    \end{aligned}\right.
\end{equation}
for $\phi \in L^2(I \times \Omega)$. The operator $K^\ast$ is well-defined according to \cite{Triggiani91,kunisch2016optimal}.\\
In order to apply Algorithm~\ref{alg:accgcg} to \eqref{eq:num_prob} we need to discretize the wave equation using a finite element method. For this purpose, consider ansatz and test spaces spanned by products of piecewise linear and continuous functions on a uniform time grid in $I$ and a spatial triangulation of $\Omega$. The adjoint equation is discretized consistently. Finally, as in Section~\ref{subsubsec:practical}, we also replace the control space by picewise constant functions on the time grid and then apply a discretized version of Algorithm~\ref{alg:accgcg} to the problem. The finite dimensional subproblems in step 5. are again solved by a semismooth Newton method. We plot the computed function~$\bar{u}$ alongside the reference~$u^*$ as well as the the norm of the optimal dual variable in Figure~\ref{fig:supportwave1}-\ref{fig:supportwave2}.
\begin{figure}[htb]
\begin{subfigure}[t]{.48\linewidth}
\centering
\includegraphics[scale=0.4]{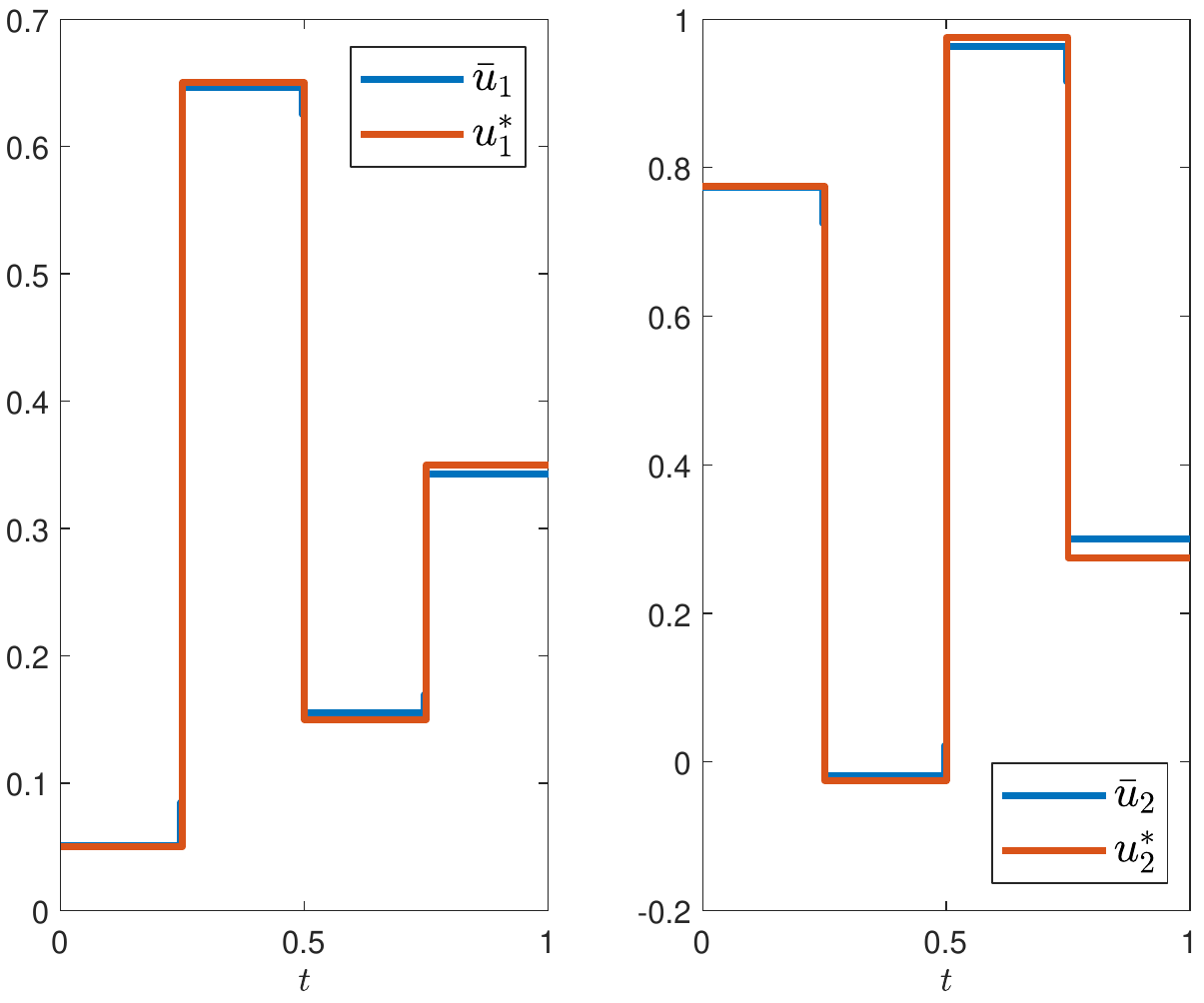}
\caption{Reference~$u^*$ and~$\optu$.}
\label{fig:supportwave1}
\end{subfigure}
\quad
\begin{subfigure}[t]{.48\linewidth}
\centering
\includegraphics[scale=0.4]{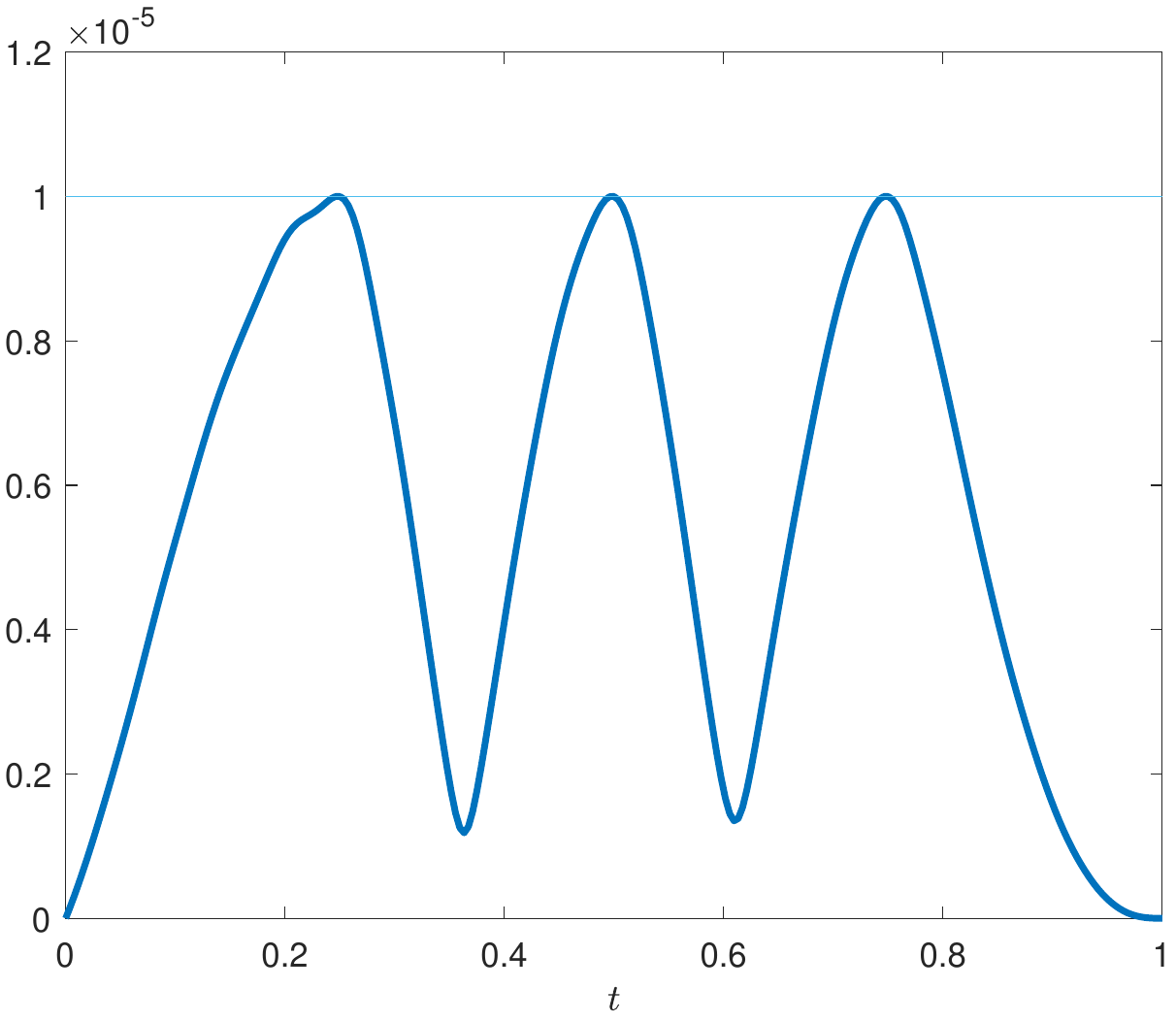}
\caption{Norm of dual variable~$\bar{p}$.}
\label{fig:supportwave2}
\end{subfigure}
\begin{subfigure}[t]{.48\linewidth}
\centering
\scalebox{.4}{
%
%
\definecolor{mycolor1}{rgb}{0.00000,0.44700,0.74100}%
\definecolor{mycolor2}{rgb}{0.85000,0.32500,0.09800}%
\begin{tikzpicture}

\begin{axis}[%
width=4.521in,
height=3.566in,
at={(0.758in,0.481in)},
scale only axis,
xmin=0,
xmax=40,
xlabel style={font=\color{white!15!black}},
xlabel={$k$},
ymode=log,
ymin=1e-09,
ymax=0.001,
yminorticks=true,
axis background/.style={fill=white},
legend style={at={(0.03,0.03)}, anchor=south west, legend cell align=left, align=left, draw=white!15!black}
]
\addplot [color=mycolor1, line width=2.0pt, mark=diamond, mark options={solid, mycolor1}]
  table[row sep=crcr]{%
0	0.000746699900576296\\
1	0.000148151251650682\\
2	9.92366391637287e-05\\
3	3.26136088932847e-05\\
4	2.49548012975089e-05\\
5	1.46215152006752e-05\\
6	1.3558595331994e-05\\
7	1.23489608401163e-05\\
8	9.55317285464985e-06\\
9	7.47582086684935e-06\\
10	5.59513954667131e-06\\
11	3.58512094747384e-06\\
12	3.45591423167006e-06\\
13	2.57268113973669e-06\\
14	2.3713101286646e-06\\
15	1.47036040804054e-06\\
16	9.71972376444888e-07\\
17	8.36808277086073e-07\\
18	6.32180017774711e-07\\
19	6.14034698622825e-07\\
20	2.93730419674924e-07\\
21	2.34990300661242e-07\\
22	2.28418348310866e-07\\
23	2.11464707082787e-07\\
24	1.64851012660129e-07\\
25	1.35860445973388e-07\\
26	1.01372643919131e-07\\
27	9.94379360378029e-08\\
28	9.86307688326851e-08\\
29	4.05419026055723e-08\\
30	3.941061054151e-08\\
31	3.74893654628793e-08\\
32	3.56036955976102e-08\\
33	3.43940409464446e-08\\
34	1.84037924776706e-08\\
35	6.45885325541303e-09\\
36	1.74292502808516e-09\\
37	1.70075256309532e-09\\
38	1.3901987638051e-09\\
39	1.0703220590437e-09\\
};
\addlegendentry{$r_j(u_k)$}

\addplot [color=mycolor2, dashdotted, line width=2.0pt]
  table[row sep=crcr]{%
0	0.000746699900576296\\
1	0.000373349950288148\\
2	0.000248899966858765\\
3	0.000186674975144074\\
4	0.000149339980115259\\
5	0.000124449983429383\\
6	0.000106671414368042\\
7	9.3337487572037e-05\\
8	8.29666556195884e-05\\
9	7.46699900576296e-05\\
10	6.78818091432996e-05\\
11	6.22249917146913e-05\\
12	5.74384538904843e-05\\
13	5.33357071840211e-05\\
14	4.97799933717531e-05\\
15	4.66687437860185e-05\\
16	4.39235235633115e-05\\
17	4.14833278097942e-05\\
18	3.92999947671735e-05\\
19	3.73349950288148e-05\\
20	3.55571381226808e-05\\
21	3.39409045716498e-05\\
22	3.24652130685346e-05\\
23	3.11124958573457e-05\\
24	2.98679960230518e-05\\
25	2.87192269452422e-05\\
26	2.76555518731961e-05\\
27	2.66678535920106e-05\\
28	2.57482724336654e-05\\
29	2.48899966858765e-05\\
30	2.40870935669773e-05\\
31	2.33343718930093e-05\\
32	2.26272697144332e-05\\
33	2.19617617816558e-05\\
34	2.13342828736085e-05\\
35	2.07416639048971e-05\\
36	2.01810783939539e-05\\
37	1.96499973835867e-05\\
38	1.91461512968281e-05\\
39	1.86674975144074e-05\\
};
\addlegendentry{$\mathcal{O}(1/(k+1))$}

\addplot [color=black, dashed, line width=2.0pt]
  table[row sep=crcr]{%
0	0.000746699900576296\\
1	0.00055173694590127\\
2	0.000407678717028778\\
3	0.000301234016595966\\
4	0.000222581971940746\\
5	0.000164465935132023\\
6	0.000121523965229545\\
7	8.97941212765843e-05\\
8	6.63489230342661e-05\\
9	4.90252538275569e-05\\
10	3.62247856173204e-05\\
11	2.67665129819934e-05\\
12	1.97777903997494e-05\\
13	1.46138196394714e-05\\
14	1.07981589519578e-05\\
15	7.97876528028398e-06\\
16	5.89551382611597e-06\\
17	4.35619824032315e-06\\
18	3.21879715130724e-06\\
19	2.37837089353744e-06\\
20	1.75737949343244e-06\\
21	1.2985286240799e-06\\
22	9.59483477448267e-07\\
23	7.08962841807618e-07\\
24	5.23853013499166e-07\\
25	3.87075264836833e-07\\
26	2.86010115027701e-07\\
27	2.11333023133476e-07\\
28	1.56154080992585e-07\\
29	1.15382331871712e-07\\
30	8.52560651859367e-08\\
31	6.29957510225246e-08\\
32	4.6547593279575e-08\\
33	3.43940409464445e-08\\
34	2.54137747900444e-08\\
35	1.87782514443352e-08\\
36	1.38752597841078e-08\\
37	1.02524366897089e-08\\
38	7.57553081614225e-09\\
39	5.59756367029569e-09\\
};
\addlegendentry{$\mathcal{O}(0.74^k)$}

\end{axis}

\begin{axis}[%
width=5.833in,
height=4.375in,
at={(0in,0in)},
scale only axis,
xmin=0,
xmax=1,
ymin=0,
ymax=1,
axis line style={draw=none},
ticks=none,
axis x line*=bottom,
axis y line*=left
]
\end{axis}
\end{tikzpicture}%
}
\caption{Residual error over $k$.}
\label{fig:reswave}
\end{subfigure}
\quad
\begin{subfigure}[t]{.48\linewidth}
\centering
\scalebox{.4}{
\centering
%
%
\definecolor{mycolor1}{rgb}{0.00000,0.44700,0.74100}%
\definecolor{mycolor2}{rgb}{0.85000,0.32500,0.09800}%
\begin{tikzpicture}

\begin{axis}[%
width=4.521in,
height=3.566in,
at={(0.758in,0.481in)},
scale only axis,
xmin=0,
xmax=40,
xlabel style={font=\color{white!15!black}},
xlabel={$k$},
ymode=log,
ymin=1e-05,
ymax=10,
yminorticks=true,
axis background/.style={fill=white},
legend style={at={(0.03,0.03)}, anchor=south west, legend cell align=left, align=left, draw=white!15!black}
]
\addplot [color=mycolor1, line width=2.0pt, mark=diamond, mark options={solid, mycolor1}]
  table[row sep=crcr]{%
0	0.332679104394192\\
1	0.249326511333728\\
2	0.213863314739937\\
3	0.0808363956895329\\
4	0.0791933887125275\\
5	0.0699669539783912\\
6	0.06573890837425\\
7	0.0681128411420446\\
8	0.0505308674073929\\
9	0.0473505339235372\\
10	0.0450667425026546\\
11	0.0329785563932157\\
12	0.0323438138565543\\
13	0.0273154427991758\\
14	0.0186783396338345\\
15	0.0122948100616028\\
16	0.0143032984968929\\
17	0.0122079994417851\\
18	0.0105042565469643\\
19	0.0105683135168499\\
20	0.0080663684277831\\
21	0.00887841638370054\\
22	0.00837596330743585\\
23	0.00827616989071222\\
24	0.00531520812913292\\
25	0.00458343819894966\\
26	0.00378027714201335\\
27	0.00374138911350134\\
28	0.00363351239395938\\
29	0.000923558232040614\\
30	0.000737379214711303\\
31	0.000792645918081725\\
32	0.000652767354943334\\
33	0.000414527637867591\\
34	0.000391115523135013\\
35	0.000211463345550766\\
36	0.000196319261898709\\
37	0.000154795381197031\\
38	0.000130477977017566\\
39	0.000162833135332364\\
};
\addlegendentry{$\|u_k-\bar{u}\|_{L^1(I)}$}

\addplot [color=mycolor2, line width=2.0pt, mark=square, mark options={solid, mycolor2}]
  table[row sep=crcr]{%
0	2.3149315799181\\
1	1.07420631588887\\
2	0.811546225104021\\
3	0.0465171627587244\\
4	0.168723885699073\\
5	0.0555216703163319\\
6	0.0884265285162211\\
7	0.158206683250322\\
8	0.196909132877082\\
9	0.126315049129613\\
10	0.099840047595086\\
11	0.136463064995679\\
12	0.133872110342756\\
13	0.120732804837727\\
14	0.135261169721061\\
15	0.121106112216473\\
16	0.0625779962837947\\
17	0.0422815525536113\\
18	0.0377370135116228\\
19	0.0387940436766026\\
20	0.0100995657736744\\
21	0.0139807217359671\\
22	0.0150050590013899\\
23	0.014068585259102\\
24	0.0142166594699145\\
25	0.00990260482136573\\
26	0.00569353254005689\\
27	0.00522727874583717\\
28	0.00512508630027586\\
29	0.000502419429263767\\
30	0.000300820768562637\\
31	5.3999576795416e-05\\
32	0.000204075274838011\\
33	0.000361548710913873\\
34	0.000115766703850717\\
35	3.70573982007016e-05\\
36	1.31104985099384e-05\\
37	4.01515037813382e-05\\
38	5.16873302713883e-05\\
39	7.0437962254033e-05\\
};
\addlegendentry{$|\|u_k '\|-\|\bar{u} '\||$}

\addplot [color=black, dashed, line width=2.0pt]
  table[row sep=crcr]{%
0	2.87974322926927\\
1	2.3149315799181\\
2	1.86089793188329\\
3	1.49591510303299\\
4	1.20251732088149\\
5	0.966664421054457\\
6	0.777069973717773\\
7	0.624661186345373\\
8	0.502144737184412\\
9	0.403657762950217\\
10	0.324487299226904\\
11	0.260844747763606\\
12	0.209684578095862\\
13	0.168558587697099\\
14	0.13549874647172\\
15	0.108923019267344\\
16	0.0875596596665962\\
17	0.0703863522375631\\
18	0.0565812909754871\\
19	0.045483852858967\\
20	0.0365629846055019\\
21	0.0293917898162098\\
22	0.0236271004110058\\
23	0.0189930547721825\\
24	0.015267896750085\\
25	0.012273363814686\\
26	0.00986615653703619\\
27	0.00793108118385816\\
28	0.00637553727318468\\
29	0.00512508630027586\\
30	0.0041198895810321\\
31	0.00331184475059146\\
32	0.00266228388802413\\
33	0.00214012311391319\\
34	0.00172037511225174\\
35	0.00138295339535089\\
36	0.00111171109143127\\
37	0.00089366825734407\\
38	0.00071839074049012\\
39	0.000577490866192023\\
};
\addlegendentry{$\mathcal{O}(0.8^k)$}

\end{axis}

\begin{axis}[%
width=5.833in,
height=4.375in,
at={(0in,0in)},
scale only axis,
xmin=0,
xmax=1,
ymin=0,
ymax=1,
axis line style={draw=none},
ticks=none,
axis x line*=bottom,
axis y line*=left
]
\end{axis}
\end{tikzpicture}%
}
\caption{Norm/$L^1$-error over~$k$.}
\label{fig:stricttopwave}
\end{subfigure}
\caption{Optimal control and convergence of relevant quantities.}
\end{figure}
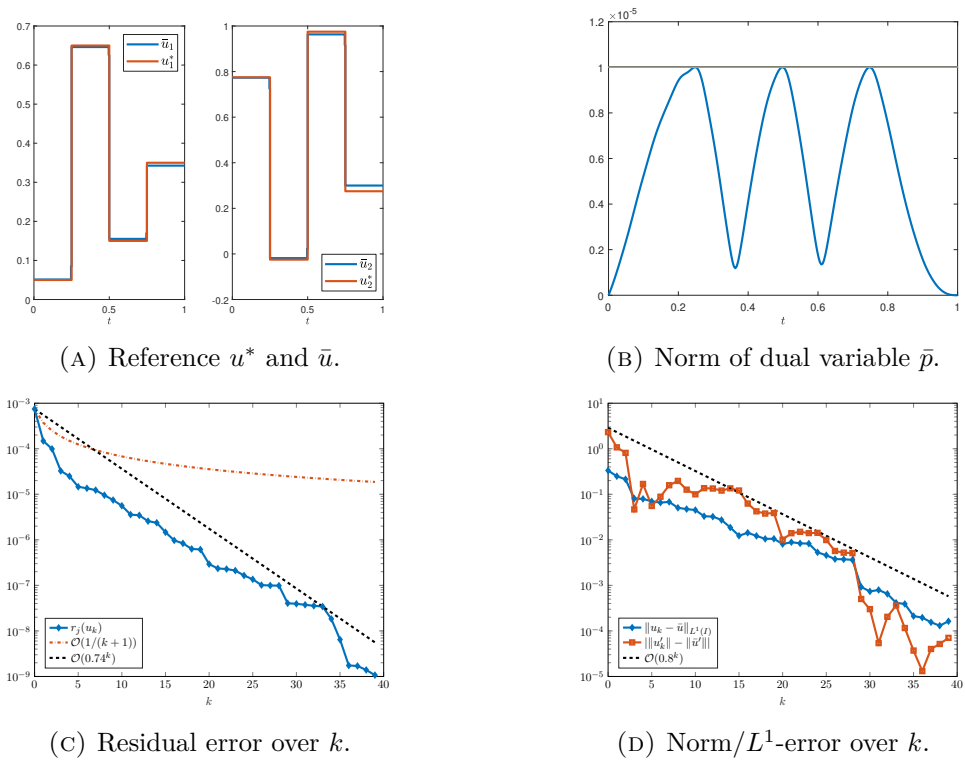
Upon a closer inspection, in contrast to the first example, we now observe local clustering of the jumps of~$\bar{u}$. More in detail,~in the vicinity of every jump of~the reference function~$u^*$,~$\optu$ admits two jumps supported on neighbouring grid nodes. Similar discretization effects for sparse deconvolution problems have been observed in~\cite{duvalspikes}. Alongside the optimal control we also report on the convergence history of the residual~$r_j(u_k)$, the~$L^1$-distance of~$u_k$ and~$\bar{u}$ as well as the error of the norms in Figure~\ref{fig:reswave} and~\ref{fig:stricttopwave}. Again we observe a linear rate of convergence for all considered quantities.
\bibliographystyle{siam}
\bibliography{BV}
\end{document}